\documentclass{amsart}
\usepackage{amssymb,amsmath,amscd,amsthm,enumerate,color,verbatim}
\usepackage[all,cmtip]{xy}
\usepackage{latexsym}
\usepackage{enumitem}
\usepackage{graphicx}
\usepackage{subfigure}
\usepackage{tikz}

\def\SL{{\mathrm {SL}}}
\def\LC{{\mathrm {LC}}}
\def\SH{{\mathrm {SH}}}
\def\loc{{\mathrm{loc}}}
\def\e{{\varepsilon}}

\def\beq{\begin{equation}}
	\def\eeq{\end{equation}}

\newcommand{\Z}{{\mathbb Z}}
\newcommand{\R}{{\mathbb R}}

\newcommand{\E}{{\mathbb E}}

\newcommand{\bbP}{{\mathbb{P}}}

\newcommand{\norm}[1]{\left\Vert#1\right\Vert}
\newcommand{\base}{K}

\newcommand{\CA}{{\mathcal A}}

\newcommand{\CF}{{\mathcal F}}

\newcommand{\CB}{{\mathcal B}}

\newcommand{\CZ}{{\mathcal Z}}

\renewcommand{\vec}[1]{\underline{#1}}
\newcommand{\set}[1]{\left\{#1\right\}}

\newtheorem{lemma}{Lemma}[section]
\newtheorem{theorem}[lemma]{Theorem}
\newtheorem{prop}[lemma]{Proposition}
\newtheorem{coro}[lemma]{Corollary}

\theoremstyle{definition}
\newtheorem{definition}[lemma]{Definition}
\newtheorem{remark}[lemma]{Remark}

\begin{document}

\title[Anderson Localization]{Schr\"odinger Operators with Potentials Generated by Hyperbolic Transformations: II.~Large Deviations and Anderson Localization}

\author{Artur Avila, David Damanik, and Zhenghe Zhang}

\address{Institut f\"ur Mathematik, Universit\"at Z\"urich, Winterthurerstrasse 190, 8057 Z\"urich, Switzerland and IMPA, Estrada D. Castorina 110, Jardim Bot\^anico, 22460-320 Rio de Janeiro, Brazil}

\email{artur@math.sunysb.edu}

\address{Department of Mathematics, Rice University, Houston, TX~77005, USA}

\email{damanik@rice.edu}

\address{Department of Mathematics, University of California, Riverside, CA-92521, USA}
\email{zhenghe.zhang@ucr.edu}

\thanks{D.\ D.\ was supported in part by NSF grants DMS--1700131 and DMS--2054752.}
\thanks{Z.\ Z.\ was supported in part by NSF grant DMS--1764154.}

\begin{abstract}
We consider discrete one-dimensional Schr\"odinger operators whose potentials are generated by H\"older continuous sampling along the orbits of a uniformly hyperbolic transformation. For any ergodic measure satisfying a suitable bounded distortion property, we establish a uniform large deviation estimate in a large energy region provided that the sampling function is locally constant or has small supremum norm. We also prove full spectral Anderson localization for the operators in question.
\end{abstract}

\maketitle

\tableofcontents

\section{Introduction}

It is a well-known phenomenon that the one-dimensional Anderson model is localized throughout the spectrum. Given a compactly supported probability measure $\nu$ on $\R$, the Anderson model is given by the discrete Schr\"odinger operator on $\ell^2(\Z)$, acting as
\begin{equation}\label{eq:dyndefpot}
[H_\omega \psi] (n) = \psi(n+1) + \psi(n-1) + V_\omega(n) \psi(n),
\end{equation}
where the random potential $V_\omega(n)$ is given by independent random variables, each distributed according to $\nu$. Formally this can be modeled by considering the compact product $\Omega = (\mathrm{supp} \, \nu)^\Z$, the shift transformation $T : \Omega \to \Omega$, $[T \omega](n) = \omega(n+1)$, and the $T$-ergodic measure $\mu = \nu^\Z$ on $\Omega$, as the potential can then be written in the form of a dynamically defined potential in the sense of \cite{DF22},
\begin{equation}\label{eq:dyndefoper}
V_\omega(n) = f(T^n \omega),
\end{equation}
where the sampling function $f : \Omega \to \R$ is given by the evaluation at the origin,
\begin{equation}\label{eq:AMsamplingfunction}
f(\omega) = \omega(0).
\end{equation}
The localization statement for the operator family $\{ H_\omega \}_{\omega \in \Omega}$ typically takes two forms. Spectral localization means that for $\mu$-almost every $\omega \in \Omega$, the operator $H_\omega$ has pure point spectrum with exponentially decaying eigenfunctions. Dynamical localization means that the unitary group associated with $H_\omega$ has exponential off-diagonal decay relative to the standard orthonormal basis $\{ \delta_n \}_{n \in \Z}$ of $\ell^2(\Z)$ uniformly in time, either in an almost sure sense or in expectation with respect to~$\mu$.

Spectral localization and dynamical localization are not equivalent for general operators, but they both hold for the Anderson model. Proving these localization statements for the Anderson model is most difficult in the case of a singular single-site distribution, for example the Bernoulli case, where $\mathrm{supp} \, \nu$ has cardinality $2$. In fact it is true that the Bernoulli case is the most difficult case to handle, as any localization proof that covers the Bernoulli case will also cover the general case. We refer the reader to \cite{CKM87} for the first proof of spectral localization in the Bernoulli case and to \cite{bucaj2, GZ20, GK21, JZ19} for recent treatments of this model, which establish both spectral and dynamical localization and are simpler and more conceptual (in that they rely on one-dimensional tools, rather than verifying the input necessary to run a multi-scale analysis).

Given this state of affairs, one might say that the one-dimensional Anderson model is completely understood. Under the hood, however, all the proofs rely crucially on the independence of the potential values, and hence only apply to the special choice \eqref{eq:AMsamplingfunction} of the sampling function $f$. From the perspective of the general theory of ergodic Schr\"odinger operators in $\ell^2(\Z)$, \cite{DF22}, it is natural to ask whether localization extends to more general choices of $f$. Indeed, this should even be expected to be the case, based on the heuristics of the situation. Alas, prior to the present work, such an extension was well outside the scope of the existing approaches. One could argue that it is the independence of the underlying entries of $\omega$ that is responsible for localization phenomena, as long as the sampling function is still sufficiently ``local'' in the sense that the values of $f(\omega)$ should be most heavily influenced by the values of $\omega_n$ for $n$ not too large. Of course this will be true for any continuous sampling function $f : \Omega \to \R$, but upon closer inspection, quantitative continuity properties of $f$ become relevant, as many exotic spectral phenomena turn out to occur for a generic continuous sampling function. 

As a very specific example of interest, consider the case of locally constant $f \in C(\Omega,\R)$, which are those for which $f(\omega)$ only depends on the values of $\omega_n$ in a fixed finite range of $n$'s. The localization problem for the associated operators $\{ H_\omega \}_{\omega \in \Omega}$ was previously inaccessible, but it will be fully covered by our work.

The approach to proving localization developed in the current series of papers applies to a much larger class of models. This is relevant as this class includes additional specific examples of interest that are formally distinct from the Anderson model, but share the feature that is crucial to our method, namely the fact that the underlying ergodic base dynamical system $(\Omega,T,\mu)$ is a subshift of finite type with the ergodic measure having the bounded distortion property. We will give precise definitions in Section~\ref{sec:mainresults} but mention here that this allows us to establish a localization result throughout the entire spectrum for Arnold's cat map. For this particular example, there was only the previous work \cite{BS} that had to exclude energy intervals of positive length from consideration. 

The general structure of our localization proof follows the outline of previous works (e.g., \cite{B, BoGo, BS}):
\begin{itemize}

\item[(i)] Establish the positivity of the Lyapunov exponent, $L(E)$, for sufficiently many energies (i.e., the exceptional set $\{ E : L(E) = 0 \}$ should be shown to be discrete).

\item[(ii)] Prove large deviation estimates for the Lyapunov exponent.

\item[(iii)] Eliminate double resonances. 

\end{itemize}

Item (i) was addressed in part I of this series \cite{ADZ}. In the present follow-up work we address items (ii) and (iii). 

To summarize, together with part I of the series, \cite{ADZ}, the present paper succeeds in establishing spectral and dynamical localization for Schr\"odinger operators that were expected to share these features with the Anderson model, but were structurally so different from it that previous methods were insufficient to make progress towards results of this kind.

We give precise definitions and statements of our results in Section~\ref{sec:mainresults}. As we can quote positivity results for the Lyapunov exponent from \cite{ADZ}, we need to establish large deviation estimates and eliminate double resonances, which then completes the proof of localization. We do the former in Section~\ref{sec:largedeviations} and the latter in Section~\ref{sec:doubleresonances}. Finally, we discuss applications to concrete examples in Section~\ref{sec:examples}.

\section{The Setting and the Main Results}\label{sec:mainresults}

\subsection{The Base Dynamical System}\label{sss:basedynamics}

Let $\mathcal{A}=\{1, 2, \ldots, \ell\}$ with $\ell \ge 2$ be equipped with the discrete topology. Consider the product space $\mathcal{A}^\Z$, whose topology is generated by the cylinder sets, which are the sets of the form
$$
[n; j_0,\cdots,j_k] = \{ \omega \in \mathcal{A}^\Z : \omega_{n+i} = j_i , \; 0 \le i \le k \}
$$
with $n \in \Z$ and $j_0, \ldots , j_k \in \mathcal{A}$. The topology is metrizable and for definiteness we fix the following metric on $\mathcal{A}^\Z$. Set
$$
N(\omega,\tilde \omega)=\max\{N\ge 0: \omega_n=\tilde \omega_n \mbox{ for all } |n|<N\},
$$
and equip $\mathcal A^\Z$ with the metric $d$ defined by
$$
d(\omega , \tilde \omega) =e^{-N(\omega,\tilde \omega)}.
$$
We consider the left shift $T : \mathcal{A}^\Z \to \mathcal{A}^\Z$ defined by $(T \omega)_n = \omega_{n+1}$ for $\omega \in \mathcal{A}^\Z$ and $n \in \Z$. Let $\mathrm{Orb}(\omega)=\{T^n\omega:\ n\in\Z\}$ be the orbit of $\omega$ under the dynamics $T$.

\begin{definition}\label{def:localsets}
	Let $\Omega \subseteq \mathcal{A}^\Z$ be a subshift of finite type and consider the topological dynamical system $(\Omega,T)$.
	
	We say that a finite word $j_0 j_1 \ldots j_k$, where $j_i \in\{1,\ldots, \ell\}$ for $0 \le i \le k$, is \emph{admissible} if it occurs in some $\omega\in\Omega$, that is, there are $\omega\in\Omega$ and $n\in\Z$ such that $\omega_{n+i}=j_i$ for all $0\le i\le k$.
	
	The \textit{local stable set} of a point $\omega \in \Omega$ is defined by
	$$
	W^s_\mathrm{loc}(\omega) = \{ \tilde \omega \in \Omega : \omega_n = \tilde \omega_n  \text{ for } n \ge 0 \}
	$$
	and the \textit{local unstable set} of $\omega$ is defined by
	$$
	W^u_\mathrm{loc}(\omega) = \{ \tilde \omega \in \Omega : \omega_n = \tilde \omega_n \text{ for } n \le 0 \}.
	$$
	
	A set is called $s$-\textit{locally saturated} (resp., $u$-\textit{locally saturated}) if it is a union of local stable (resp., local unstable) sets of the form above.
	
	For each $j \in \mathcal{A}$ and each pair of points $\omega, \tilde \omega\in [0;j]$, we denote the unique point in $W^u_{\mathrm{loc}}(\omega)\cap W^s_{\mathrm{loc}}(\tilde\omega)$ by $\omega\wedge\tilde \omega$.
\end{definition}

\subsection{Measures With the Bounded Distortion Property}\label{sss:local_product}

Let the subshift $\Omega$ be equipped with the Borel $\sigma$-algebra and let $\mu$ be a probability measure on $\Omega$ that is ergodic with respect to $T$. We define
\begin{align*}
	&\Omega^+=\{(\omega_n)_{n\ge 0}: \omega\in\Omega\}\\
	&\Omega^-=\{(\omega_n)_{n\le 0}: \omega\in \Omega\}
\end{align*}
to be the spaces of one-sided right and left infinite sequences, respectively, associated with $\Omega$. Metrics for $\Omega^\pm$ can be defined in a way similar to the definition of the metric for $\CA^\Z$ above. Abusing notation slightly, we still let $d$ denote their metrics. Let $\pi^{+}$ be the projection from $\Omega$ to $\Omega^+$ and $\mu^+=\pi^+_*(\mu)$ be the pushforward measure of $\mu$ on $\Omega^+$. Similarly, we let  $\pi^-$ be the projection to $\Omega^-$ and $\mu^-$ be the pushforward measure on $\Omega^-$. Let $T_+$ be the left shift operator on $\Omega^+$ and $T_-$ be the right shift on $\Omega^-$. For $n\ge 0$, we let $[n;j_0, \ldots, j_k]^+$ denote the cylinder sets in $\Omega^+$; for $n \le -k$, we let $[n;j_0,\ldots, j_k]^-$ denote the cylinder sets in $\Omega^-$. Let $\omega^\pm$ denote points in $\Omega^\pm$, respectively.

For simplicity, for each $1\le j\le \ell$, we set $\mu_j=\mu|_{[0;j]}$ and $\Omega_j = \Omega \cap [0;j]$. Similarly, we set $\mu^\pm_j=\mu^\pm|_{[0;j]^\pm}$ and $\Omega^\pm_j=\Omega^\pm\cap[0;j]^\pm$, respectively.

Note that we do not have $\Omega=\Omega^-\times\Omega^+$. However, for each $1\le j\le \ell$, we have a natural homeomorphism:
$$
P: \Omega_j\to \Omega^-_j\times\Omega^+_j\mbox{ where } P(\omega)=(\pi^-\omega,\pi^+\omega).
$$
Thus, abusing the notation a bit, we may just write $\Omega_j=\Omega^-_j\times \Omega^+_j$. Moreover, we have for all $\omega \in \Omega$,
\beq\label{eq:local_su_set}
(\pi^+)^{-1}(\pi^+\omega)=W^s_{\mathrm{loc}}(\omega),\ (\pi^-)^{-1}(\pi^-\omega)=W^u_{\mathrm{loc}}(\omega).
\eeq

\begin{definition}
	We say $\mu$ has a \emph{local product structure} if there is a $\psi:\Omega\to\R_+$ such that for each $1\le j \le \ell$, $\psi\in L^1(\Omega_j,\mu^-_j\times\mu^+_j)$ and
	\beq\label{eq:localproduct}
	d\mu_j=\psi \cdot d(\mu^+_j\times \mu^-_j).
	\eeq
\end{definition}

If $\mu$ has a local product structure, then its topological support $\mathrm{supp} \, \mu$ is a subshift of finite type (see, e.g., \cite[Lemma~1.2]{BGV}) and hence, without loss of generality, we will assume throughout that the measure $\mu$ has full support in $\Omega$. For results concerning positivity of the LE, a local product structure is usually sufficient, see e.g. \cite{ADZ}. However, for further results such as large deviation estimates, we will need the measure $\mu$ to obey a quantitative version of local product structure, which is defined as follows.

\begin{definition}
	We say that $\mu$ satisfies the \textit{bounded distortion property} if there is $C \ge 1$ such that for all cylinders $[n;j_0,\ldots,j_{k}]$ and $[l;,i_{0},\ \ldots, j_{m}]$ in $\Omega$, where $l> n+k$ and $[n;j_0,\ldots,j_{k}]\cap [l;,i_{0},\ldots, i_{s}]\neq\varnothing$, we have
	\beq\label{eq:bdd}
	C^{-1} \le \frac{\mu \left( [n;j_0,\ldots,j_{k}]\cap [l;i_{0},\ldots, i_{s}] \right)}{\mu \left( [n;j_0,\ldots,j_{k}] \right) \cdot \mu \left( [l;i_{0},\ldots,i_{s}] \right)} \le C.
	\eeq
\end{definition}

The bounded distortion property implies local product structure; see, for example, \cite[Section 2.1]{ADZ}. Equilibrium states of H\"older continuous potentials have the bounded distortion property; compare, for example, \cite[Lemma 3.4]{ADZ}. It is a standard result that if an invariant measure $\mu$ of a $C^2$ transitive Anosov diffeomorphism (or a $C^2$ transitive, uniformly expanding map) is absolutely continuous with respect to the volume measure, then it is an equilibrium state of a H\"older continuous potential; see, for example, \cite{bowen2}. In particular, it includes hyperbolic endomorphisms on the $d$-dimensional torus, for all $d\ge1$, such as the doubling map on $\R/\Z$ or Arnold's cat map on $(\R/\Z)^2$, where the ergodic measure is simply taken to be the Lebesgue measure . It is also well-known that Markov chains with Markov measures have locally constant potentials, which are clearly H\"older continuous. 

Recall that  $(\Omega, T)$ is \emph{topologically mixing} if for any pair of nonempty open sets $U, V\subset\Omega$, there is an $N\ge 1$ such that $T^n(U)\cap V\neq\varnothing$ for all $n\ge N$. Following \cite[Section 2.1.4]{ADZ}, we can assume without loss of generality that $(\Omega, T)$ is topologically mixing in Section~\ref{sec:largedeviations}. Concretely, by the spectral decomposition theorem for hyperbolic basic sets, we can decompose $\Omega$ as $\Omega = \bigsqcup^{s}_{l=1}\Omega^{(l)}$ for some $s\ge 1$ and for closed subsets $\Omega^{(l)}$, so that the following holds true: $T(\Omega^{(l)})=\Omega^{(l+1)}$ for $1 \le l < s$, $T(\Omega^{(s)})=(\Omega^{(1)})$, and $T^s|\Omega^{(l)}$ is a topologically mixing subshift of finite type for each $1\le l\le s$. Moreover, the normalized restriction $\mu^{(l)}$ of $\mu$ to $\Omega^{(l)}$ is a $T^s$-invariant ergodic, fully supported probability measure with local product structure or bounded distortion property, provided the same property is true for $\mu$ on $\Omega$. 

For a cocycle map $A : \Omega \to \mathrm{SL}(2,\R)$ with $\|A\|_\infty\le\Gamma$, we consider $A_s : \Omega^{(l)} \to \mathrm{SL}(2,\R)$ as $A_s(\omega)$ (see \eqref{eq:A_n} below for the definition of $A_s$), which may be considered a cocycle map defined over the base dynamics $T^s : \Omega^{(l)} \to \Omega^{(l)}$. Clearly, $L(A_s, \mu^{(l)}) > 0$ for some $1 \le l \le s$ implies that $L(A,\mu) > 0$. Moreover, large deviation estimates on each $(T^s, \Omega^{(l)}, \mu^{(l)})$ imply large deviation estimates on $(T, \Omega, \mu)$ as follows. For $n\ge 1$, we may write $n=sk+r$ for some $0\le r<s$. Then for each $\omega\in\Omega$, 
\[
\|A_{r}(T^{ks}\omega)\|^{-1}\cdot\|A_{ks}(\omega)\|\le \|A_{n}(\omega)\|\le \|A_{r}(T^{ks}\omega)\|\cdot \|A_{ks}(\omega)\|.
\]
Hence for each $\varepsilon>0$, there is a $n_0=n_0(\e, \Gamma)$ such that for all $n\ge n_0$, we have
\[
\left|\frac1n\log \|A_n(\omega)\|-\frac{1}{ks}\log\|A_{ks}(\omega)\|\right|<\frac\e2.
\]
In particular, we have
\[
\left|\frac1n\log \|A_n(\omega)\|-L(E)\right|>\varepsilon\Rightarrow \left|\frac1k\log \|A_{ks}(\omega)\|-sL(E)\right|>\frac s2\varepsilon,
\]
which implies for all $n\ge n_0$ that
\[
\left\{\omega: \left|\frac1n\log \|A_n(\omega)\|-L(E)\right|>\varepsilon\right\}\subset \bigcup^s_{l=1}\left\{\omega\in\Omega^{(l)}: \left|\frac1k\log \|A_{ks}(\omega)\|-sL(E)\right|>\frac s2\varepsilon\right\}.
\]
Exponential decay of the measure of the set at the right side above is clearly a direct consequence of the LDT of $(T^s, A_s)$ on each $(T^s, \Omega^{(l)}, \mu^{(l)})$.

The main consequence of topological mixing that we will use is the following. There is $r_0 \in\Z_+$ so that for all $[k;j_0,\ldots, j_n] \subset \Omega$ and all $[l;i_0,\ldots, i_m] \in \Omega$, where $l-(k+n) \ge r_0$, we have
\beq\label{eq:nonempty_intersection}
[k;j_0,\ldots, j_n]\cap[l;i_0,\ldots, i_m]\neq\varnothing.
\eeq
\eqref{eq:nonempty_intersection} could be an easy consequence of the specification property of such systems; see, for example, \cite[Proposition 2.7]{ADZ}. In Section~\ref{sec:largedeviations}, for our $(\Omega,T,\mu)$, we let $r_0$ be a number satisfying \eqref{eq:nonempty_intersection}, which will be used in the proof of Lemma~\ref{l.cesaro1}.

\subsection{The Associated Operator Family}\label{sss:operators}

In this paper, we are mainly concerned with the Anderson localization phenomenon for one-dimensional discrete Schr\"odinger operators $H_\omega$ in $\ell^2(\Z)$ acting by
\beq\label{eq:operator}
[H_\omega u](n) = u(n+1) + u(n-1) + V_\omega(n) u(n).
\eeq
Here we assume $\Omega$ to be any compact metric space, $T:\Omega\to\Omega$ a homeomorphism, and $f : \Omega \to \R$ continuous. We consider potentials $V_\omega : \Z \to \R$ defined by $V_\omega(n) = f(T^n \omega)$ for $\omega \in \Omega$ and $n \in \Z$. Let 
\[
\sigma(H_\omega)=\{E: H_\omega-E \mbox{ does not have a bounded inverse}\}
\]
be the spectrum of $H_\omega$. Choose a $T$-ergodic measure $\mu$. Then there is a compact set $\Sigma$ such that $\sigma(H_\omega)=\Sigma$ for $\mu$-a.e. $\omega$. We say that $H_\omega$ has (spectral) \emph{Anderson localization} if it has pure point spectrum with exponentially decaying eigenfunctions.

A continuous map $A : \Omega \to \mathrm{SL}(2,\R)$ gives rise to the cocycle $(T,A) : \Omega \times \R^2 \to \Omega \times \R^2$, $(\omega , v) \mapsto (T \omega , A(\omega) v)$. For $n \in \Z$, we let $(T,A)^n = (T^n , A_n)$. In particular, we have
\begin{equation}\label{eq:A_n}
A_n(\omega)=
\begin{cases}
	A(T^{n-1}\omega) \cdots A(\omega),\ & n\ge 1;\\
	I_2, &n=0;\\
	[A_{-n}(T^n\omega)]^{-1}, \ &n\le -1,
\end{cases}
\end{equation}
where $I_2$ is the identity matrix. The Lyapunov exponent is given by
\begin{align*}
	L(A,\mu)
	& = \lim_{n \to \infty} \frac{1}{n} \int \log \|A_n(\omega)\| \, d\mu(\omega) \\
	& = \inf_{n \ge 1} \frac{1}{n} \int \log \|A_n(\omega)\| \, d\mu(\omega)\ge 0.
\end{align*}
By Kingman's subaddive ergodic theorem, we have
$$
\lim_{n \to \infty} \frac{1}{n}\log \|A_n(\omega)\| = L(A,\mu)
$$
for $\mu$-a.e. $\omega$. By linearity and invertibility of each $A(\omega)$, we can projectivize the second component and consider $(T,A) : \Omega \times \R\bbP^1 \to \Omega \times \R\bbP^1$.

Spectral properties of the operators $H_\omega$ can be investigated in terms of the behavior of the solutions to the difference equation
\begin{equation}\label{e.eve}
	u(n+1) + u(n-1) + V_\omega(n) u(n) = E u(n), \quad n \in \Z,
\end{equation}
with $E \in \R$. These solutions in turn can be described with the help of the Schr\"odinger cocycle $(T,A^E)$ with the cocycle map $A^{E}:\Omega\to\SL(2,\R)$ being defined as
$$
A^E(\omega) = A^{(E-f)}(\omega):= \begin{pmatrix} E-f(\omega) & -1 \\ 1 & 0 \end{pmatrix},
$$
where we often leave the dependence on $f:\Omega \to \R$ implicit as it will be fixed most of the time. Such cocycles describe the transfer matrices associated with Schr\"odinger operators with dynamically defined potentials. Specifically, $u$ solves \eqref{e.eve} if and only if
$$
\begin{pmatrix} u(n) \\ u(n-1) \end{pmatrix} = A^{E}_n(\omega) \begin{pmatrix} u(0) \\ u(-1) \end{pmatrix}, \quad n \in \Z.
$$
We set $L(E)=L(A^E,\mu)$ and let $\CZ_f=\{E: L(E)=0\}$.

\subsection{The Main Results}\label{sss:results}

As mentioned in the introduction, there are two key ingredients to a localization proof -- positivity of the Lyapunov exponent and large deviation estimates. Let us formulate precisely the two statements that are needed.

\begin{definition}\label{d:ple_uld}
	Let $I\subset \R$. We say $A^E$ has PLE on $I$ if 
	\[
	\inf_{E\in I}L(E)>0.
	\]
	We say $A^E$ has ULD on $I$ if for every $\varepsilon>0$, there are constants $C,c>0$, depending only on $\varepsilon$ and $f$, such that it for all $E \in I$ and  $n\ge 1$, we have
	\[
	\mu\big\{\omega\in\Omega: \big|\frac1n\log\|A^E_n(\omega)\|-L(E)\big|>\varepsilon\big\}<Ce^{-cn}.
	\]
\end{definition}

Since PLE may be quoted from \cite{ADZ}, our main tasks will be to establish ULD and then to deduce the desired localization results from PLE and ULD. In this paper, we shall show ULD for the following two classes of sampling functions. They are both subsets of $C(\Omega,\R)$ and they require quantitative control on the dependence of the function of entries of the input sequence that are far away from the origin.

The first class is the following, where there is in fact no dependence on entries that are sufficiently far away from the origin.

\begin{definition}
We say that $f:\Omega\to\R$ is \emph{locally constant} if there exists an $n_0 \in \Z_+$ such that for each $\omega\in\Omega$, $f(\omega)$ depends only on the cylinder set $[-n_0;\omega_{-n_0}, \ldots, \omega_{n_0}]$. In other words, $f$ is constant on each such cylinder set. Let us denote by $\LC$ the set of all locally constant $f:\Omega\to\R$.
\end{definition}
\medskip

\begin{remark}
(a) Clearly, any locally constant $f$ is Lipschitz continuous and hence $\LC \subset C^\alpha(\Omega,\R)$ for any $0<\alpha\le 1$, which is defined below.
\\[1mm]
(b)	We can similarly define locally constant cocycles $A:\Omega\to\SL(2,\R)$ and, abusing notation slightly, denote by $\LC$ the set of all such cocycles as well. Clearly, $f\in \LC$ implies that $A^E\in \LC$ for all $E$.
\end{remark}
	
The second class consists of sufficiently small H\"older continuous functions $f:\Omega\to\R$. We write $C^\alpha(\Omega,\R)$, $0<\alpha\le 1$ for the space of real-valued $\alpha$-H\"older continuous functions. That is,
\[
\sup_{\omega\neq \omega'}\frac{|f(\omega)-f(\omega')|}{d(\omega,\omega')^\alpha}<\infty.
\] 
Likewise, we can define the space of $\alpha$-H\"older continuous cocycles $C^\alpha(\Omega, \SL(2,\R))$.

\begin{definition}
	We say that $A \in C^\alpha (\Omega , \mathrm{SL}(2,\R))$ is \emph{fiber bunched} if there exists $N \ge 1$ such that for every $\omega \in \Omega$, we have
	\beq\label{eq:bunching2}
	\|A_{N}(\omega)\|^2 < e^{\alpha N}.
	\eeq
	Equivalently, there is $\theta < \alpha$ such that $\|A_{N}(\omega)\|^2 < e^{\theta N}$ for every $\omega \in \Omega$. 
	\end{definition}
	\medskip
	\begin{remark}
		As explained in \cite[Remark 7.5]{ADZ}, $A^E$ is fiber bunched for all $E\in [-2-\lambda_0,2+\lambda_0]\supset\Sigma_f$, provided $\|f\|_\infty < \lambda_0 := \frac{(e^{\frac\alpha2}-1)^2}{9}$. 
		\end{remark}

Hence, we make the following definition:
\begin{definition}\label{def:sh}
We say $f\in C^\alpha(\Omega,\R)$ is \emph{globally fiber bunched} if $\|f\|_\infty<\lambda_0$ and denote by $\SH$ the set of all such $f$.
\end{definition}

\medskip
	
	Suppose $T$ has a fixed point on $\Omega$ and $\mu$ is a $T$-ergodic measure that has a local product structure. Then the following holds true if $f \in \LC \cup \SH$. Let $\CF_f$ be the set of energies where $A^E$ has an $su$-state (see Section~\ref{ss:LDT} below for a detailed description of $su$-states). By \cite[Proposition 5.9]{ADZ}, $\CZ_f\subset\CF_f$ and  by \cite[Lemma 5.11]{ADZ}, $\CF_f$ is finite. By \cite[Theorem 2.8]{backes}, $L(E)$ is continuous on $\R$. Then for any compact interval $J$ and any $\eta>0$, we have PLE on
\begin{equation}	\label{e.jeta}
J_\eta:=J\setminus B(\CF_f,\eta), 
\end{equation}
which consists of a finite number of connected compact intervals. Here, $B(\CF_f,\eta)$ denotes the open $\eta$-neighborhood of $\CF_f$.

The first main result of the present paper addresses the ULD property:
	
\begin{theorem}\label{t:uld}
Let $\Omega$ be a subshift of finite type and $\mu$ be a $T$-ergodic measure that has the bounded distortion property. Assume that $T$ has a fixed point. Let $f\in \LC \cup \SH$ be nonconstant. Then there is a connected compact interval $J\supset\Sigma_f$ such that $A^{E}$ satisfies \emph{ULD} on $J_\eta$ for all $\eta>0$.
\end{theorem}

\begin{remark}
Large deviation estimates in similar contexts were previously obtained by several authors, see for example \cite{duarteklein, gouezel, park}. In particular, \cite{duarteklein} proved a local uniform version, which enabled them to obtain H\"older continuity of the Lyapunov exponent. However, they all assumed certain typical conditions of the cocycles that were first introduced by \cite{BGV, BV}, while in the context of Theorem~\ref{t:uld}, we do not necessarily have these typical conditions for $A^E$, $E\in J_\eta$. Indeed, typical conditions were used in \cite{BGV, BV} to prove positivity of the Lyapunov exponent (or simplicity of the Lyapunov spectrum in case of higher dimensional cocycles), whereas the proof of positivity of the Lyapunov exponent on $J_\eta$ in \cite{ADZ} does not use any perturbation argument. What \cite{ADZ} uses instead is a certain analyticity argument together with certain aspects of inverse spectral theory. In fact, Theorem~\ref{t:uld} is stronger than the previous results in the sense that the typical conditions of \cite{BGV, BV} imply the uniqueness of the $u$-state, which is the condition needed to prove our ULD. Although we do not pursue a local uniform version of LDT in the space of $C^\alpha(\Omega, \mathrm{SL}(2,\R))$, our proof does imply that without any extra work since uniformity is a direct consequence of the uniqueness of the $u$-state. We refer the reader to Section~\ref{ss:LDT} for a more detailed description. Moreover, our techniques, which are different from those of the previous works, can be further developed to obtain stronger results concerning ULD which we shall explore in the third paper of this series.
	\end{remark}
	
Our second main result deduces localization from PLE and ULD:

\begin{theorem}\label{t:main}
	 Let $\Omega$ be a subshift of finite type and $\mu$ be a $T$-ergodic measure that satisfies the bounded distortion property. Let $f\in C^\alpha(\Omega,\R)$. Let $I\subset \R$ on which we have \emph{PLE} and \emph{ULD}. Then, $H_\omega$ has spectral localization on $I$ for $\mu$-almost every $\omega\in\Omega$.
\end{theorem}

\begin{remark}
In fact, our proof implies exponential dynamical localization of $H_\omega$ on $I$ for $\mu$-a.e. $\omega$ which is known to be stronger than the spectral localization. More concretely, for any  $\mu$-a.e. $\omega$, all $\varepsilon > 0$, and $0 < \beta < \gamma=\inf_{E\in I}L(E)$, there is a constant $\widetilde C = \widetilde C_{\omega,\beta,\varepsilon} > 0$ such that
\begin{equation}\label{eq:edl}
	\sup\limits_{t \in \R}
	|\langle\delta_n, e^{-itH_\omega}P_{I,\omega}\delta_m\rangle|
	\leq
	\widetilde C e^{\epsilon |m|} e^{-\beta|n-m|}
\end{equation}
for all $m,n \in \Z$, where $P_{I,\omega}$ denotes the spectral projection of $H_\omega$ to $I$.
\end{remark}

Corollaries~\ref{t:locally_constant} and ~\ref{t:fiber_bunched} below are direct consequences of Theorems~\ref{t:uld} and ~\ref{t:main}. We state them separately to facilitate reference to one of them.

\begin{coro}\label{t:locally_constant}
Let $(\Omega, \mu)$ be as in Theorem~\ref{t:uld}. If $f \in \LC$ is nonconstant, then for $\mu$-almost every $\omega\in\Omega$, $H_\omega$ has full spectral localization.
	\end{coro}

\begin{coro}\label{t:fiber_bunched}
Let $(\Omega, \mu)$ be as in Theorem~\ref{t:uld}. If $f\in \SH$ is nonconstant, then for $\mu$-almost every $\omega\in\Omega$, $H_\omega$ has full spectral localization.
	\end{coro}
	
In the same spirit, for ease of reference, let us formulate the following sample applications to specific base dynamics.

	\begin{coro}\label{cor:doubling}
		Consider $T:(\R/\Z)^m\to(\R/\Z)^m$ where $T$ is the doubling map if $m=1$ or Arnold's cat map if $m=2$. Let $\mu$ be the Lebesgue measure. Let $f:(\R/\Z)^m\to\R$ be H\"older continuous and nonconstant, and consider the potentials given  by $V_\omega(n)=\lambda f(T^n\omega)$. Then there is a $\lambda_0 > 0$ such that for all $0<\lambda \le\lambda_0$, $H_{\omega}$ has full spectral localization for $\mu$-a.e. $\omega\in(\R/\Z)^m$.
		\end{coro}
		
		\begin{coro}\label{cor:markov}
			  Let $(\Omega, T, \mu)$ be a Markov chain with a Markov measure $\mu$. Suppose $T$ has a fixed point and $f \in \LC$ is nonconstant. Consider the potentials given  by $V_\omega(n)=\lambda f(T^n\omega)$. Then for all $\lambda>0$, $H_{\omega}$ has full spectral localization for $\mu$-a.e. $\omega$.
			\end{coro}
			
			\begin{remark}
				We wish to point out that according to Definition~\ref{def:sh} in Corollary~\ref{t:fiber_bunched}, $f\in \SH$ simply means $f\in C^\alpha(\Omega,\R)$ and $\|f\|_\infty\le \lambda_0$, where
				$$
				\lambda_0=\frac{(e^\frac\alpha2-1)^2}{9}.
				$$
				But in Corollary~\ref{cor:doubling} the $\lambda_0$ might be different due to the coding process that converts the smooth hyperbolic systems to their corresponding subshifts of finite type. Take the doubling map for example, by \cite[Remark 7.4]{ADZ} we may choose $\lambda_0$ to be
				$$
				\lambda_0=\frac{(2^\frac\alpha2-1)^2}{9}.
				$$
				In any case, $\lambda_0$ is not too small if $\alpha$ is not small. 

			\end{remark}

\begin{remark}
Corollary~\ref{cor:markov} covers in particular our initial motivating example from the introduction: for the Bernoulli shift with the Bernoulli measure, any nonconstant locally constant sampling function will produce an operator family that is almost surely spectrally localized throughout the spectrum. One should remark, though, that dynamical localization will only hold away from a discrete set of exceptional energies. There are examples where such energies are indeed present and may produce transport at a rate that is understood; compare \cite{ADZ, DT1, DT2, JS, JSS} for relevant discussions of this phenomenon. We also emphasize that Corollary~\ref{cor:doubling} fully covers the main results of \cite{BS} and fills in the energy intervals on which \cite{BS} was inconclusive regarding the localization statement.
\end{remark}

\section{Large Deviation Estimates -- Proof of Theorem~\ref{t:uld}}\label{sec:largedeviations}
In this section, we prove Theorem~\ref{t:uld}. Thus, we focus on sampling functions $f\in \LC$ or $f\in \SH$. For simplicity, we denote the projectivized action of $B\in \SL(2,\R)$ on $\R\bbP^1$ by $B\cdot v$, $v\in\R\bbP^1$.

\subsection{Stable and Unstable Holonomies}\label{ss:LDT}

Let us recall the following central concept.

\begin{definition}\label{d:holonomies}
	A \textit{stable holonomy} $h^s$ for a continuous $A:\Omega\to\SL(2,\R)$ is a family of matrices
	 $$
	 \{h^s_{\omega , \omega'}\in \SL(2,\R) : \omega\in\Omega, \omega'\in W^s_\loc(\omega)\},
	 $$
such that
	\begin{itemize}
		
		\item[(i)] $h^s_{\omega' , \omega''}h^s_{\omega, \omega'} = h^s_{\omega, \omega''}$ and $h^s_{\omega, \omega} = \mathrm{id}$,
		
		\item[(ii)] $A(\omega') h^s_{\omega, \omega'} = h^s_{T \omega, T \omega'}  A(\omega)$,
		
		\item[(iii)] $(\omega, \omega') \mapsto h^s_{\omega, \omega'}$ is uniformly continuous for all $\omega\in \Omega$ and all $\omega'\in W^u_\loc(\omega)$.
		
	\end{itemize}
	An \textit{unstable holonomy} $\{h^u_{\omega , \omega'} : \omega\in\Omega, \omega'\in W^u_\loc(\omega)\}$ for $A$ is a stable holonomy for $A^{-1}$ over $T^{-1}$.
\end{definition}
If $A$ is locally constant or fiber bunched, then it is a standard result (see e.g. the proof of \cite[Lemma 4.1]{ADZ}) that the stable and unstable holonomies can be defined as
\begin{equation}\label{e.holonomiesdef}
	H^s_{\omega,\omega'} = \lim_{n \to \infty} A_n(\omega')^{-1} A_n(\omega), \quad H^u_{\omega,\omega'} = \lim_{n \to \infty} A_{-n}(\omega')^{-1} A_{-n}(\omega)
\end{equation}
for $\omega , \omega'$ in the same stable and unstable sets, respectively. Moreover, the convergence is uniform for all $\omega\in \Omega$ and all $\omega'\in W^s_\loc(\omega)$ and all $\omega'\in W^u_\loc(\omega)$, respectively. The properties (i)--(iii) for $H^s_{\omega, \omega'}, H^u_{\omega, \omega'}$ follow directly from the construction. Holonomies that arise from \eqref{e.holonomiesdef} are called \emph{canonical holonomies} of $A$.

\begin{definition}\label{d:inv_meas_on_fiber}
	Suppose we are given a $(T,A)$-invariant probability measure $m$ on $\Omega \times \R\bbP^1$ that projects to $\mu$ in the first component. A \emph{disintegration} of $m$ along the fibers is a measurable family $\{m_\omega: \omega\in\Omega\}$ of conditional probability measures on $\R\bbP^1$ such that $m = \int m_\omega \, d\mu(\omega)$, that is,
	$$
	m(D)=\int_\Omega m_\omega(\{z\in\R\bbP^1:(\omega,z)\in D\}) \, d\mu(\omega)
	$$
	for each measurable set $D\subset \Omega\times \R\bbP^1$.
\end{definition}

By Rokhlin's disintegration theorem, such a disintegration exists. Moreover, $\{\tilde m_\omega:\omega\in\Omega\}$ is another disintegration of $m$ if and only if $m_\omega=\tilde m_\omega$ for $\mu$-almost every $\omega\in\Omega$. By a straightforward calculation one checks that $\{ A(\omega)_* m_{\omega} : \omega \in \Omega\}$ is a disintegration of $(T,A)_*m$. In particular, the facts above imply that $m$ is $(T,A)$-invariant if and only if $A(\omega)_* m_\omega = m_{T\omega}$ for $\mu$-almost every $\omega\in\Omega$.

Such a measure $m$ will be called an $s$-\emph{state} (resp., a $u$-\emph{state}) if it is in addition invariant under the stable (resp., unstable) holonomies. That is, the disintegration $\{m_\omega:\omega\in\Omega\}$ satisfies $(h^s_{\omega,\omega'})_* m_\omega = m_{\omega'}$ for $\mu$-almost every $\omega\in\Omega$ and every $\omega'\in W^s_\loc(\omega)$ (resp., $(h^u_{\omega,\omega'})_* m_\omega = m_{\omega'}$ for $\mu$-almost every $\omega\in\Omega$ and every $\omega'\in W^u_\loc(\omega)$). In this case, we say that $\{m_\omega\}$ is $s$-\emph{invariant} (resp., $u$-\emph{invariant}). A measure that is both an $s$-state and a $u$-state is called an $su$-\emph{state}.

\begin{definition}\label{def:indep_of_future}
	We say that a function defined on $\Omega$ \textit{only depends on the future} (resp., \emph{past}) if it is constant on every local stable (resp., unstable) set.
\end{definition}

We can use the stable or unstable holonomy to reduce $A$ so that it is constant on local unstable or stable sets as follows. For each $1\le j\le \ell$, fix a choice of $\omega^{(j)}\in[0;j]$. We define $\varphi(\omega)=\omega\wedge\omega^{(\omega_0)}$, which depends only on the past. We define a new cocycle map as follows. 
$$
\tilde A^{-1}(\omega):=H^u_{T^{-1}\omega, \varphi(T^{-1}\omega)}\cdot A^{-1}(\omega)\cdot H^u_{\varphi(\omega), \omega}.
$$
It is clear that $(T^{-1}, \tilde A^{-1})$ is conjugate to $(T^{-1}, A^{-1})$ via the unstable holonomy. By property (iii), $\tilde A^{-1}$ is continuous. By conditions (i) and (ii), we have that
\begin{align*}
	\tilde A^{-1}(\omega)
	&=H^u_{T^{-1}\omega, \varphi(T^{-1}\omega)}\cdot A^{-1}(\omega)\cdot H^u_{\varphi(\omega), \omega}\\
	&= H^u_{T^{-1}\omega, \varphi(T^{-1}\omega)}\cdot H^u_{T^{-1}\varphi(\omega),T^{-1}\omega}\cdot A(\varphi(\omega))\\
	&= H^u_{T^{-1}\phi(\omega), \varphi(T^{-1}\omega)}\cdot A(\varphi(\omega)),
\end{align*}
which implies that $\tilde A(\omega)$ depends only on the past. Similary, we can conjugate $(T,A)$ to $(T,\bar A)$ via the stable holonomy so that $\bar A$ depends only the future. 

The maps $H^{u,E}_{\omega,\omega'}$ are continuous and uniformly bounded for all $E$ in any compact set. It is then straightforward to see that large deviation estimates as stated in Definition~\ref{d:ple_uld} are preserved under such a conjugacy. Hence, from now on, we may without loss of generality assume that $A^E$ depends only the past. In particular, for such matrices, we have
$$
H^u_{\omega,\omega'}=I_2
$$ 
for all $\omega\in\Omega$ and all $\omega'\in W^u_\loc(\omega)$.

In the context of Theorem~\ref{t:uld}, $f\in \LC$ or $f\in \SH$. Hence, $A^E$ is fiber bunched throughout some connected compact interval $J$ containing the almost sure spectrum $\Sigma_f$. So we may let $H^{s,E}_{\omega, \omega'}$ and $H^{u,E}_{\omega,\omega'}$ be their canonical stable and unstable holonomies, respectively. 

\subsection{Reduction of the Uniform LDT} 

Let $F^E:\Omega\times\R\bbP^1\to\Omega\times\R\bbP^1$ be given by $F^E(\omega,v)=(T\omega, A^E(\omega)\cdot v)$. For a continuous function $\varphi:\Omega\times\R\bbP^1$, we let $S_n^E(\varphi)(\omega,v)=\sum^{n-1}_{k=0}\varphi((F^E)^k(\omega,v))$ be its Birkhoff sum with respect to the dynamics $F^E$. Thus a measure $\nu$ on $\Omega\times\R\bbP^1$ is $(T,A^E)$-invariant if and only if $(F^E)_*\nu=\nu$. Throughout this section, we may sometimes leave the dependence of $A^E$, $F^E$, and $H^{s(u), E}$ on $E$ implicit if it is clear from the context.

\begin{lemma}\label{p:unique_u-state}
	For every $E\in J_\eta$, $A^E$ has a unique $u$-state. Moreover, the unique $u$-state is continuous in $E$ with respect to the weak*-topology.
	\end{lemma}
\begin{proof}
Let $E\in J_\eta$ be arbitrary. By the definition of $J_\eta$ in \eqref{e.jeta} and the discussion surrounding \eqref{e.jeta}, we have $L(E)>0$. Thus, by Oseledec's multiplicative ergodic theorem, $A$ has stable and unstable directions for $\mu$-almost every $\omega$. We consider the Dirac measures $\delta^{s}(\omega)$ and $\delta^{u}(\omega)$ on $\R\bbP^1$, where $\delta^{s}(\omega)$ and $\delta^{u}(\omega)$ are Dirac masses concentrated on the stable and unstable directions, respectively. We can define two $F^E$-invariant probability measures $m^{s}$ and $m^{s}$ on $\Omega\times \R\bbP^1$ as follows:	
	$$
	m^{s} = \int \delta^{s}(\omega) \, d\mu(\omega)\mbox{ and }m^{u} = \int \delta^{u}(\omega) \, d\mu(\omega).
	$$

The invariance of $m^{s}$ and $m^{u}$ follows from the $(T,A)$-invariance of the stable and unstable directions, respectively. Property (ii) of Definition~\ref{d:holonomies} implies
	$$
	A_n(\omega') H^s_{\omega, \omega'} = H^s_{T^n \omega, T^n\omega'}  A_n(\omega),
	$$
	which in turn implies that the length of the vector $A_n(\omega')H^s_{\omega,\omega'}\vec v$ decays exponentially for any vector $\vec v\in \mathrm{supp} \; \delta^s(\omega)$. Thus $H^{s}_{\omega,\omega'}\delta^s(\omega)=\delta^s(\omega')$, which implies that $m^s$ is an $s$-state (with disintegration $\{\delta^s(\omega)\}_{\omega\in\Omega}$). Applying the same argument to $(T^{-1}, A^{-1})$ and the unstable holonomies $H^u_{\omega,\omega'}$, we obtain that $m^u$ is a $u$-state (with disintegration $\{\delta^u(\omega)\}_{\omega\in\Omega}$). Moreover, it is not difficult to see that any $F$-invariant measure $m$ on $\Omega\times\R\bbP^1$ is a convex combination of $m^u$ and $m^s$; see, for example, \cite{backes}. 
	
We claim that $m^u$ is the unique $u$-state of $A$. Indeed, suppose there is a $u$-state $m\neq m^u$. Then $m=tm^u+(1-t)m^s$ for some $0\le t< 1$. Thus
	$$
	m^s=\frac{1}{1-t}m+\frac{t}{1-t}m^u,
	$$
	which implies $m^s$ is a $u$-state, hence an $su$-state, of $A^E$. Thus, $E \in \CF_f$, which contradicts our choice of $E$. 
	
To show the continuity of $m^u$ in $E$, we let $m^{u,E}$ denote the unique $u$-state of $A^E$ for each $E\in J_\eta$. Let $\{E_n\}$ be a convergent sequence in $J_\eta$ and $\lim\limits_{n\to\infty}E_n=E_0\in J_\eta$. To show that $m^{u, E_n}$ converges to $m^{u, E_0}$ in the weak*-topology, it suffices to show that the limit of any convergent subsequence is $m^{u,E_0}$. Without loss of generality, we may just assume that $\{m^{u,E_n}\}$ is convergent with the limit denoted by $m$; we need to show that $m=m^{u, E_0}$. Note that
	$$
	A^{E_n}(\omega)-A^{E_0}(\omega)=\begin{pmatrix}E_n-E_0 &0 \\ 0 & 0\end{pmatrix},
	$$
	which clearly implies
$$
\lim\limits_{n\to\infty}\|A^{E_n}-A^{E_0}\|_{\infty}= 0.
$$
 Note by our assumption, $H^{u,E}_{\omega,\omega'}=I_2$ for all parameters. Thus, by a special case of  \cite[Lemma 4.3]{backes}, it follows that $m$ is a $u$-state with respect to $(T,A^{E_0})$. By the uniqueness of the $u$-state, we have $m=m^{u, E_0}$, as desired.
\end{proof}

By the definition of a local unstable set, we may think of $W^u_\loc$ as a map on $\Omega^-$:
$$
W^u_\loc (\omega^-)=\{\omega\in\Omega: \pi^-(\omega)=\omega^-\}.
$$ 
Thus we may decompose $\Omega$ as a disjoint union of local unstable sets as follows:
$$
\Omega=\bigsqcup_{\omega^-\in \Omega^-} W^u_\loc(\omega^-).
$$

A \emph{disintegration} of $\mu$ along the local unstable sets is a measurable family 
$$
\{\mu^u_{\omega^-}: \mbox{ a probablity measure on } W^u_\loc(\omega^-)\}_{\omega^-\in\Omega^-}
$$ 
such that $\mu= \int_{\Omega^-} \mu^u_{\omega^-}d\mu^-(\omega^-)$. In other words,
\beq\label{eq:disint_mu}
\mu(D)=\int_{\Omega^-} \mu^u_{\omega^-}\big(\{\omega\in D: \pi^-(\omega)=\omega^-\}\big) \, d\mu^-(\omega^-)
\eeq
for each measurable set $D\subset \Omega$. Again by Rokhlin's disintegration theorem, such a disintegration exists. For $v\in\R\bbP^1$, we let $\vec v$ be a unit vector in the direction of $v$. Then we have:

\begin{theorem}\label{t:ldt_loc-u-mfld}
For every $\varepsilon>0$, there exist $C, c>0$ such that uniformly for all $(\omega^-, v)\in\Omega^-\times\R\bbP^1$ and all $E\in J_\eta$, we have
$$
\mu^u_{\omega^-}\left\{\omega\in W^u_\loc(\omega^-):  \bigg|\frac1n\log\|A^E_n(\omega)\vec v\|-L(E)\bigg|>\varepsilon\right\}<Ce^{-cn}.
$$
\end{theorem}
 Theorem~\ref{t:uld} is an easy consequence of Theorem~\ref{t:ldt_loc-u-mfld} as follows.
 \begin{proof}[Proof of Theorem~\ref{t:uld}]
 	By Theorem~\ref{t:ldt_loc-u-mfld} and the disintegration \eqref{eq:disint_mu}, we clearly have
 	\[\mu\left\{\omega\in\Omega:  \bigg|\frac1n\log\|A^E_n(\omega)\vec v\|-L(E)\bigg|>\varepsilon\right\}<Ce^{-cn}
 	\]
 	uniformly for all $E\in J_\eta$. Choosing $\vec v=(0,1)$ and $(1,0)$, it then follows from a standard computation (see, e.g., the proof of \cite[Theorem 3.1]{bucaj2}) that the estimate above implies
 		\[\mu\left\{\omega\in\Omega:  \bigg|\frac1n\log\|A^E_n(\omega)\|-L(E)\bigg|>\varepsilon\right\}<Ce^{-cn}
 	\]
 	uniformly for all $E\in J_\eta$, as desired.
 	\end{proof} 
 
 On the other hand, Theorem~\ref{t:ldt_loc-u-mfld} follows from the next theorem. Recall that for a continuous function $\varphi$ on $\Omega\times\R\bbP^1$, we denote by $S^E_n(\varphi)(\omega,v)=\sum^{n-1}_{k=0}\varphi((F^E)^k(\omega,v))$ its Birkhoff sum with respect to the dynamics $F^E$. 

\begin{theorem}\label{t:ldt_additive}
For every $\varphi\in C^\alpha(\Omega\times\R\bbP^1,\R)$, $\varepsilon>0$, and $E_0\in J_\eta$, there exist $C, c, r>0$ such that uniformly for all $(\omega^-, v)\in\Omega^-\times\R\bbP^1$ and all $E\in (E_0-r, E_0+r)\cap J_\eta$, we have that
	$$
	\mu^u_{\omega^-}\left\{\omega\in W^u_\loc(\omega^-): \bigg|\frac1nS^E_n(\varphi)(\omega,v)-\int\varphi d\, m^{u,E}\bigg|>\varepsilon\right\}<Ce^{-cn}.
	$$
\end{theorem}

Theorem~\ref{t:ldt_loc-u-mfld} follows from Theorem~\ref{t:ldt_additive}:

\begin{proof}[Proof of Theorem~\ref{t:ldt_loc-u-mfld}]
	 Let $\varphi^{E}(\omega,v)=\log\|A^{E}(\omega)\vec v\|$. It is a direct computation to see that $\varphi^E\in C^\alpha(\Omega\times\R\bbP^1)$ since $f\in C^\alpha(\Omega,\R)$. Applying Theorem~\ref{t:ldt_additive} to $\varphi^{E_0}$ and the given $\varepsilon>0$, we obtain some positive constants $r_0, c, C$ such that 
	 $$
	 \mu^u_{\omega^-}\left\{\omega\in W^u_\loc(\omega^-):  \bigg|\frac1nS^E_n(\varphi^{E_0})(\omega,v)-\int\varphi^{E_0} d\, m^{u,E}\bigg|>\frac\varepsilon3\right\}<Ce^{-cn}.
	 $$
	 uniformly for all $(\omega^-, v)\in\Omega^-\times\R\bbP^1$ and all $E\in (E_0-r_0, E_0+r_0)\cap J_\eta$.
	 
	 For any $E_0\in J_\eta$, let $\{E_n\}$ be a sequence converging to $E_0$. Then it is clear that 
 $$
 \lim\limits_{n\to\infty}\|\varphi^{E_n}-\varphi^{E_0}\|_\infty=0,
 $$
 where the supremum norm is taken over $\Omega\times\R\bbP^1$. Thus for the given $\varepsilon>0$, there exists $r_1>0$ such that for all $E$ with $|E-E_0|<r_1$, we have
\begin{align*}
	&\left \|\frac1nS^E_n(\varphi^E)-\frac1nS^E_n(\varphi^{E_0})\right\|_{0}<\frac\varepsilon3 \mbox{ and } \\ 
	&\left|\int\varphi^E \, dm^{u,E}-\int\varphi^{E_0}\, dm^{u, E}\right|<\frac\varepsilon3.
\end{align*}

Setting $r=\min\{r_0, r_1\}$ and combining the estimates above, we obtain
 $$
 \mu^u_{\omega^-}\left\{\omega\in W^u_\loc(\omega^-):  \bigg|\frac1nS^E_n(\varphi^{E})(\omega,v)-\int\varphi^{E} d\, m^{u,E}\bigg|>\varepsilon\right\}<Ce^{-cn},
 $$
uniformly for all $(\omega^-, v)\in\Omega^-\times\R\bbP^1$ and all $E\in (E_0-r, E_0+r)\cap J_\eta$. 

By compactness of $J_\eta$, and changing $C, c$ if necessary, we get the uniformity in $E\in J_\eta$. This yields the desired conclusion of Theorem~\ref{t:ldt_loc-u-mfld}. Indeed, we just need to note that
	\begin{align*}
		\frac1nS^E_n(\varphi^E)(\omega,v)
		&=\frac1n\sum^{n-1}_{k=0}\varphi^E(T^k\omega,\vec{A^{E}_k(\omega)v})\\
		&=\frac1n\sum^{n-1}_{k=0}\log\bigg\|A^{E}(T^{k}\omega)\frac{A^{E}_k(\omega)\vec v}{\|A^{E}_k(\omega)\vec v\|})\bigg\|\\
		&=\frac1n\log\|A^{E}_n(\omega)\vec v\|,
	\end{align*}
	and by Birkhorff ergodic theorem, we have
	$$ 
	\int\varphi^E dm^{u,E}=L(E).
	$$
This concludes the proof.
\end{proof}

\subsection{Proof of Theorem~\ref{t:ldt_additive}} 

In this subsection, we shall prove Theorem~\ref{t:ldt_additive}. Let $C, c$ be some universal constants with $0<c<1<C$. We also let $\omega^{-,j}\in \Omega^-_j$ for each $j\in \{1,\ldots, \ell\}$.

\begin{lemma}\label{l.cesaro1}
Consider any $\omega^{-,j}\in\Omega^-_j$ and let $\nu^u$ be a probability measure on $W^u_\loc(\omega^{-,j})$ such that
	\begin{equation}\label{eq:eqiv_to_mu+}
	C^{-1} \le \frac{d((\pi^+)_* \nu^u)}{d\mu^+_j} \le C.
	\end{equation}
  Then, we have
	\begin{equation}\label{eq:nu_to_mu}
	\frac{1}{n} \sum_{k=0}^{n-1} T^k_* \nu^u \to \mu
	\end{equation}
	in the weak-$*$ topology as $n \to \infty$, uniformly in $\omega^{-,j}$, $1\le j\le \ell$, and $\nu^u$.
\end{lemma}

\begin{proof}
	Let $\nu_j=\mu^-_j \times (\pi^+)_* \nu^u$ be a measure on $[0;j]$. Let $\nu=\frac{1}{\mu^-_j(\Omega^-_j)}\nu_j$, which is a probability measure. By \eqref{eq:eqiv_to_mu+}, we have
	$$
	C^{-1} \le \frac{d(\mu^-_j\times (\pi^+)_* \nu^u)}{d(\mu^-_j\times\mu^+_j)} \le C.
	$$
	By the bounded distortion property of $\mu$, it is straightforward to see (see, e.g., \cite[Section 2.1.2]{ADZ}) that 
	$$
	C^{-1} \le \frac{d\mu_j}{d(\mu^-_j\times\mu^+_j)} \le C.
	$$
	Combining the two estimates above, we obtain
\[
	C^{-1} \le \frac{d\nu_j}{d\mu_j} \le C,
\]
	which implies that for all $n\ge 1$, we have
	\begin{equation}\label{eq:equiv_to_mu_ave}
			C^{-1} \le \frac{d\big(\sum^{n-1}_{k=0}T^k_*\nu_j\big)}{d\big(\sum^{n-1}_{k=0}T^k_*\mu_j\big)} \le C.
		\end{equation}

By \eqref{eq:nonempty_intersection}, for every cylinder set $[i;j_0,\cdots, j_s]$, we have $T^{-k}([i;j_0,\cdots, j_s])\cap[0;j]\neq\varnothing$ for all $k$ sufficiently large. In particular, we can consider such $k$'s so that $k+i>0$. By the bounded distortion property, we then have 
\begin{align*}
(T^k_*\mu_j)([n;j_0,\cdots, j_s])&=\mu(T^{-k}([i;j_0,\cdots, j_s]\cap[0;j])\\
&=\mu([i+k;j_0,\cdots, j_s]\cap[0;j])\\
&\ge C^{-1}\mu([i+k;j_0,\cdots, j_s])\cdot \mu([0;j])\\
&\ge C^{-1}\mu([i+k;j_0,\cdots, j_s]).
\end{align*}
Thus by $T$-invariance of $\mu$, for every cylinder set  $[i;j_0,\cdots, j_s]$, we have for each large $n$ that
\[
C^{-1}\le\frac{\frac1n\sum^{n-1}_{k=0}(T^k_*\mu_j)([i;j_0,\cdots, j_s])}{\mu([i;j_0,\cdots, j_s])}\le C.
\] 
Thus by \eqref{eq:equiv_to_mu_ave}, the estimates above hold true if we replace $\mu_j$ by $\nu=\frac{1}{\mu^-_j(\Omega^-_j)}\nu_j$. This further implies that if $\nu'$ is a limit point of the sequence $\{\frac1n\sum^{n-1}_{k=0}(T^k_*\nu)\}_n$, then we have 
\begin{equation}\label{eq:equiv_to_mu}
	C^{-1} \le \frac{d\nu'}{d\mu} \le C.
\end{equation}

	 Since $\Omega$ is a compact metric space, by the  Stone-Weierstrass theorem the set $C(\Omega,\R)$ is separable, which in turn implies that the weak-$*$ topology on the set of Borel probability measures on $\Omega$ is metrizable. For instance, such a metric may be chosen as 
	\[
	\rho(\nu,\nu'):=\sum^{\infty}_{s=1}2^{-s}\left|\int f _s\, d\nu-\int f_s\, d\nu'\right|,
	\]
	where $\{f_s: s\in \Z_+\}$ is a dense subset of the set $\{f\in C(\Omega,\R): \|f\|_\infty\le 1\}$.
	Note that by the Banach-Alaoglu theorem, the weak-$*$ topology on the set of Borel probability measures on any topological space is compact. Combining the two facts above, we see that the the weak-$*$ topology on the set of Borel probability measures on $\Omega$ is sequentially compact.  
	
	Now let $\nu^{(n)} = \frac{1}{n} \sum_{k=0}^{n-1} T_*^k \nu$ for any $n\in \Z_+$. By the discussion above and \eqref{eq:equiv_to_mu}, the sequence $\{ \nu^{(n)} \}$ has a weak-$*$ accumulation point $\nu^{(0)}$ that is equivalent to $\mu$. It is clear that $\nu^{(0)}$ is also $T$-invariant. Set $g=\frac{d\nu^{(0)}}{d\mu}$, which is $T$-invariant since both $\nu^{(0)}$ and $\mu$ are $T$-invairant. By ergodicity of $\mu$, there is a constant $c$ such that $g(\omega)=c$ for $\mu$-a.e. $\omega$. Since both $\nu^{(0)}$ and $\mu$ are probability measures, we must have $c=1$, which implies $\nu^{(0)}=\mu$.
	
	 Thus, by uniqueness of the accumulation point, we have $\nu^{(n)} \to \mu$ as $n\to\infty$ in the weak-$*$ topology. Applying uniqueness of the limit and metrizability of the set of Borel probability measures, we have that the convergence is uniform in $\nu$. That is, the convergence is uniform in the choices of $\nu^u$ satisfying \eqref{eq:eqiv_to_mu+} and the choices of $\omega^{-,j}$. Indeed, assume for the sake of contradiction that the convergence is not uniform. Then there is $\varepsilon_0>0$ and a strictly increasing sequence $\{n_l\}_{l\ge 1}$ in $\Z_+$ with the following property. For every $l$, there is $\nu(l)$ satisfying \eqref{eq:equiv_to_mu} such that
	\[
	\rho(\nu^{(n_l)}(l),\mu)\ge\varepsilon_0,
	\]
	where $\nu^{(n)}(l):=\frac{1}{n} \sum_{k=0}^{n-1} T_*^k \nu(l)$. Without loss of generality, we may assume $\{\nu^{(n_l)}(l)\}_{l\ge 1}$ is convergent and the limit is $\widetilde\nu$. Thus $\widetilde\nu$ satisfies \eqref{eq:equiv_to_mu} and $\rho(\widetilde\nu, \mu)\ge\varepsilon_0$. Moreover, $\widetilde \nu$ is $T$-invariant. Indeed, 
	\begin{align*}
		\rho(T_*\nu^{(n_l)}(l), \nu^{(n_l)}(l))&=\frac1{n_l}\sum^{\infty}_{s=1}2^{-s}\left|\sum^{n_l-1}_{k=0}\int f_s\, d(T^{k+1}_*\nu(l))-\int f_s\, d(T^k_*\nu(l))\right|\\
		&=\frac1{n_l}\sum^{\infty}_{s=1}2^{-s}\left|\int f_s\, d(T^{n_l}_*\nu(l))-\int f_s\, d(\nu(l))\right|\\
		&\le \frac{2}{n_l},
\end{align*}
which implies 
\[
T_*\widetilde \nu=\lim_{l\to\infty} T_*\nu^{(n_l)}(l)=\lim_{l\to\infty}\nu^{(n_l)}(l)=\widetilde\nu.
\]
Thus the same argument above showing $\nu^{(0)}=\mu$ implies that $\widetilde\nu=\mu$, which contradicts $\rho(\widetilde\nu, \mu)\ge \varepsilon_0$. In summary, so far we have obtained the estimate \eqref{eq:nu_to_mu} except that we only did so for $\nu$. The final step is to  replace $\nu$ by $\nu^u$. To this end, it suffices to prove the following claim.  

If we let $\nu^{u,(n)} = \frac{1}{n} \sum_{k=0}^{n-1} T_*^k \nu^u$, then for any continuous $\varphi$, we claim that 
	\[
	\lim_{n\to\infty}\left(\int \varphi \, d\nu^{u,(n)} - \int \varphi \, d\nu^{(n)} \right)= 0,
	\] 
	where the convergence is also uniform in $\omega^{-,j}$ and $\nu^u$. Indeed, we have
\[
\int \varphi \, d\nu^{u,(n)} - \int \varphi \, d\nu^{(n)}=\frac{1}{n}\sum^{n-1}_{k=1}\left(\int\varphi \, d (T^k_*\nu^u)-\frac1{\mu^-_j(\Omega^-_j)}\int \varphi \, d(T^k_*\nu_j) \right),
\]
which implies that it suffices to show $\int\varphi \, d (T^k_*\nu^u)-\frac1{\mu^-_j(\Omega^-_j)}\int \varphi \, d(T^k_*\nu_j)\to 0 $, uniformly in $\omega^{-}$ and $\nu^u$. To this end, we have
\begin{align*}
	&\int\varphi \, d (T^k_*\nu^u)-\frac1{\mu^-_j(\Omega^-_j)}\int \varphi \, d(T^k_*\nu_j)\\
	&=\int\varphi\circ T^k\, d \nu^u-\frac1{\mu^-_j(\Omega^-_j)}\int \varphi\circ T^k \, d\nu_j\\
	&=\int\varphi\circ T^k\, d \nu^u-\frac1{\mu^-_j(\Omega^-_j)}\int \varphi\circ T^k \, d\big(\mu^-_j \times (\pi^+)_* \nu^u\big)\\
	&=\int_{W^u_\loc(\omega^{-,j})}\varphi\circ T^k\, d \nu^u-\int_{\Omega^+_j} \left(\frac1{\mu^-_j(\Omega^-_j)}\int_{\Omega^-_j}\varphi\circ T^k(\omega^-,\omega^+) \, d\mu^-_j (\omega^-)\right)\, d\big((\pi^+)_* \nu^u\big)(\omega^+)\\
	&=\int_{\Omega^+_j} \left(\varphi\circ T^k(\omega^{-,j}, \omega^+)-\frac1{\mu^-_j(\Omega^-_j)}\int_{\Omega^-_j}\varphi\circ T^k(\omega^-,\omega^+) \, d\mu^-_j (\omega^-)\right)\, d\big((\pi^+)_* \nu^u\big)(\omega^+).
	\end{align*}
Clearly, we have
\begin{align*}
	&\left|\varphi\circ T^k(\omega^{-,j}, \omega^+)-\frac1{\mu^-_j(\Omega^-_j)}\int_{\Omega^-_j}\varphi\circ T^k(\omega^-,\omega^+) \, d\mu^-_j (\omega^-)\right|\\
	&=\left|\frac1{\mu^-_j(\Omega^-_j)}\int_{\Omega^-_j}\varphi\circ T^k(\omega^{-,j}, \omega^+)-\varphi\circ T^k(\omega^-,\omega^+) \, d\mu^-_j (\omega^-)\right|\\
	&\le \frac1{\mu^-_j(\Omega^-_j)}\int_{\Omega^-_j}\left|\varphi\circ T^k(\omega^{-,j}, \omega^+)-\varphi\circ T^k(\omega^-,\omega^+) \right|\, d\mu^-_j (\omega^-),
	\end{align*}
which tends to $0$ uniformly in $\omega^+$ as $k\to \infty$ by uniform continuity of $\varphi$ and the fact $d(T^k(\omega^{-,j},\omega^+) ,T^k(\omega^-,\omega^+) )\to 0$ as $k\to\infty$. The estimates above imply that $\int\varphi \, d (T^k_*\nu^u)-\frac1{\mu^-_j(\Omega^-_j)}\int \varphi \, d(T^k_*\nu_j)\to 0 $, uniformly in $\omega^{-,j}$ and $\nu^u$, as desired.
\end{proof}

\begin{lemma}\label{l.cesaro2}
Let $E\in J_\eta$ and $\omega^{-,j} \in \Omega^-_j$. Suppose $m$ is a probability measure on $W^u_\loc(\omega^{-,j}) \times \R\bbP^1$, whose projection $\nu^u$ to $W^u_\loc(\omega^{-,j})$ satisfies the assumptions of Lemma~\ref{l.cesaro1}. Then, we have
	$$
	\frac{1}{n} \sum_{k=0}^{n-1} (F^{E})^k_* m \to m^{u,E}
	$$
	in the weak-$*$ topology, uniformly in $\omega^{-,j}$, $E\in J_\eta$, and such choices of $m$. 
\end{lemma}

\begin{proof}
	By compactness of $\Omega\times\R\bbP^1$ and the same reasoning as in the proof of Lemma~\ref{l.cesaro1}, we see that weak-$*$ accumulation points of $\frac{1}{n} \sum_{k=0}^{n-1} (F^E)^k_* m$ exist. 
	
	Consider such an accumulation point and denote it by $\widetilde m$.  Without loss of generality, we may just assume that $\frac{1}{n} \sum_{k=0}^{n-1} (F^E)^k_* m$ converges to $\widetilde m$. Clearly, $\widetilde m$ is invariant under $F^E$ and it projects to an accumulation point of  $\frac{1}{n} \sum_{k=0}^{n-1} T^k_* \nu^u$ in the first component. By our assumption on $\nu^u$ and Lemma~\ref{l.cesaro1}, $\frac{1}{n} \sum_{k=0}^{n-1} T^k_* \nu^u$ converges to $\mu$ as $n\to\infty$. That is, $\widetilde m$ projects to $\mu$ in the first coordinate. 
	
	Next, we show that any disintegration $\{\widetilde m_\omega\}_{\omega\in\Omega}$ of $\widetilde m$ is invariant under the unstable holonomy. By our assumption that $A$ depends only on the past, this is equivalent to having $\widetilde m_\omega=\widetilde m_{\omega'}$ for $\omega$ and $\omega'$ in the same local unstable set. To this end, we let $\widetilde m^-=(\pi^-\times id)_*(\widetilde m)$, which is a measure on $\Omega^-\times \R\bbP^1$. First, we note that $A^E$ naturally descends to a map on $\Omega^-$ since it depends only on the past. That is, $A^E(\omega)=\tilde A^E(\pi^-\omega)$ for some map $\tilde A^E:\Omega^-\to\SL(2,\R)$. Abusing notation slightly, we still let $A^E$ denote the map on $\Omega^-$. Let $F^{E}_-$ denote the projectivized action of $(T_-, (A^E)^{-1})$ on $\Omega^-\times \R\bbP^1$ where $T_-$ is the right shift on $\Omega^-$. Then we naturally have 
	$$
	(\pi^-\times id)\circ (T,A^E)^{-1}=(T_-,(A^{E})^{-1})\circ (\pi^-\times id),
	$$
	which implies $(\pi^-\times id)\circ (F^E)^{-1}=F^{E}_-\circ (\pi^-\times id)$. We claim that $\widetilde m^-$ is invariant under $F^E_-$. Indeed,
	\begin{align*}
		(F^{E}_-)_*\widetilde m^-&=(F^{E}_-)_*(\pi^-\times id)_*(\widetilde m)\\
		&=(F^{E}_-\circ(\pi^-\times id))_*\big(\lim\limits_{n\to\infty} \frac{1}{n} \sum_{k=0}^{n-1} (F^E)^k_* m\big)\\
		&=((\pi^-\times id)\circ (F^E)^{-1})_*\big(\lim\limits_{n\to\infty} \frac{1}{n} \sum_{k=0}^{n-1} (F^E)^k_* m\big)\\
		&=(\pi^-\times id)_*\big(\lim\limits_{n\to\infty} \frac{1}{n} \sum_{k=0}^{n-1} (F^E)^{k-1}_* m\big)\\
		&=(\pi^-\times id)_*\widetilde m\\
		&=\widetilde m^-.
		\end{align*}
	Let $\{\widetilde m^-_{\omega^-}\}$ be a disintegration of $\widetilde m^-$. By $F^E_-$-invariance, we then have 
	$$
	(A^{E}(T_-\omega^-)^{-1})_*\widetilde m^-_{\omega^-}=\widetilde m^-_{T_-\omega^-},
	$$
	or equivalently,
	\begin{equation}\label{eq:F_-invariance}
		(A^{E}_{-1}(\omega^-))_*\widetilde m^-_{\omega^-}=\widetilde m^-_{T_-\omega^-}.
	\end{equation}
	By a special case of \cite[Lemma 3.4]{AV}, the disintegration $\{\widetilde m_\omega\}$ of  $\widetilde m$ can be recovered from the disintegration $\{\widetilde m^-_{\omega^-}\}$ of $\widetilde m^-$ via 
	$$
	\widetilde m_\omega=\lim\limits_{n\to\infty}(A^E_{-n}(\pi^-(T^n\omega)))_*\widetilde m^-_{\pi^-(T^n\omega)}
	$$
    for $\mu$-a.e. $\omega$. But \eqref{eq:F_-invariance} and the fact $T_-\circ\pi^-=\pi^-\circ T^{-1}$ imply that
    $$
(A^E_{-n}(\pi^-(T^n\omega)))_*\widetilde m^-_{\pi^-(T^n\omega)}=\widetilde m^-_{
	\pi^-\omega}
	$$
	for all $n\ge 1$. Thus we have
	$$
		\widetilde m_\omega=\widetilde m^-_{\pi^-\omega},
	$$
	which stays constant for all $\omega'\in W^u_\loc(\pi^-\omega)$. This concludes the proof that $\widetilde m$ is a $u$-state. Therefore, by uniqueness of the $u$-state, it must be equal to $m^{u,E}$. Uniform convergence follows again from uniqueness of the limit as in the proof of Lemma~\ref{l.cesaro1}.
\end{proof}

The final piece needed to prove Theorem~\ref{t:ldt_additive} is the following large deviation estimate for martingale difference sequences; see, for example, \cite{alon, BS}.

\begin{lemma}\label{l.azuma}
	Let $\{d_n\}_{n\ge 1}$ be a martingale difference sequence adapted to some filtration. That is, $\{d_n\}_{n\ge 1}$ is a sequence of variables such that
	\begin{equation}\label{e.azumass1}
		\E (d_{n+1} | \mathcal{F}_n) = 0,
	\end{equation}
	where $\mathcal{F}_n$ is the $\sigma$-algebra generated by $d_1 , \ldots , d_n$. Then for every $\varepsilon>0$, there is a $c>0$ such that for all $N\ge 1$,
	\begin{equation}\label{e.azumass3}
		\bbP\left[\bigg|\sum^{N}_{1}d_n\bigg|>\varepsilon N^{\frac12}\left(\sum^N_{1}\|d_n\|^2_\infty\right)^\frac12\right]<e^{-c\varepsilon^2N}.
	\end{equation}
\end{lemma}

\begin{remark}
	Note that if there is $a$ such that 
	\begin{equation}\label{e.azumass2}
		\sup_{n \ge 1} \|d_n\|_\infty< a,
	\end{equation}
	then \eqref{e.azumass3} becomes
	\begin{equation}\label{e.azumaconc}
		\bbP \Big( \Big| \frac{1}{N} \sum^N_1 d_n \Big| > \varepsilon \Big) \le e^{-\frac{c\varepsilon^2}{2a^2} n},
	\end{equation}
	which is already known; compare, for example, \cite{A} or \cite{lesigne}. In fact, $c$ may be taken to be $1$ in \eqref{e.azumaconc}.
\end{remark}

We are now ready to prove Theorem~\ref{t:ldt_additive}.

\begin{proof}[Proof of Theorem~\ref{t:ldt_additive}]
	
	By the bounded distortion property of $\mu$, there exists $C\ge 1$ such that for each $1\le j\le \ell$ and $\mu^-$-almost every $\omega^{-,j}\in\Omega^-_j$, 
	\begin{equation}\label{eq:eqiv_to_mu+0}
		C^{-1} \le\frac{d(\pi^+_* \mu^{u}_{\omega^{-,j}})}{d\mu^+_j}\le C.
	\end{equation}
	Since the estimates on a zero $\mu^-$-measure subset of $\Omega^-$ are negligible when we pass the estimates to $\mu$ via \eqref{eq:disint_mu}, we may without loss of generality assume that \eqref{eq:eqiv_to_mu+0} holds uniformly for all $\omega^{-,j}\in\Omega^-_j$. Thus $\mu^u_{\omega^{-,j}}$ satisfies the assumption of Lemma~\ref{l.cesaro1}. Choose any $v\in\R\bbP^1$ and let $m$ be the lift of $\mu^u_{\omega^{-,j}}$ to $W^u_\loc(\omega^{-,j})\times \{ v\}$. Then $m$ satisfies the assumption of Lemma~\ref{l.cesaro2}. Since $\pi^+T=T_+\pi^+$, we have $\pi^+_*T_*\mu^u_{\omega^{-,j}}=(T_+)_*\pi^+_*\mu^u_{\omega^{-,j}}$ which implies
	\[
	C^{-1}\le\frac{d(\pi^+_*T_*\mu^u_{\omega^{-,j}})}{d((T_+)_*\mu^+_j)} \le C.
	\]
	By $T_+$-invariance of $\mu^+$, for $ji$ admissible we have $((T_+)_*\mu^+_j)|_{[0;i]^+}=\mu^+_i$. Hence, the above estimate implies that $\frac{1}{\mu^u_{\omega^{-,j}}([0:ji])}T_*\mu^u_{\omega^{-,j}}|_{[0;i]}$ is a probability measure that satisfies \eqref{eq:eqiv_to_mu+0}, hence the assumption of Lemma~\ref{l.cesaro1}. Note that 
	\[
	F^E_*m|_{[0;i]\times\R\bbP^1}=(F^E|_{[0;j,i]\times\R\bbP^1})_*m=(F^E|_{(W^u_\loc(w^{-,j})\cap[0;j,i])\times\R\bbP^1})_*m
	\]
	is the lift of $T_*\mu^u_{\omega^{-,j}}|_{[0:i]}$ to  $(T(W^u_\loc(\omega^{-,j}))\cap[0:i])\times \{A^E(\omega)v\}$, where $A^E(\omega)v$ is a constant since $\omega\in W^u_\loc(\omega^{-,j})\cap[0;j,i]$ and $A^E$ depends only on the past. Thus 
	\[
	\frac{1}{\mu^u_{\omega^{-,j}}([0:ji])}F^E_*m|_{[0:i]\times\R\bbP^1}=	\frac{1}{m([0:j,i]\times\R\bbP^1)}F^E_*m|_{[0:i]\times\R\bbP^1}
	\] 
	satisfies the assumptions of Lemma~\ref{l.cesaro2} as well. By induction, if $ji_1\cdots i_s$ is admissible, then
	\[
	\frac{1}{\mu^u_{\omega^{-,j}}([0:j,i_1,\ldots,i_s])}T^s_*\mu^u_{\omega^{-,j}}|_{[0;i_1,\ldots, i_s]}
	\] 
	satisfies  the assumption of Lemma~\ref{l.cesaro1}. Likewise,
	\[
	(F^E)^s_*m|_{[0;i_1,\ldots,i_s]}=\big((F^E)^s|_{W^u_\loc(w^{-,j})\cap[0;j,i_1,\ldots,i_s]}\big)_*m
	\] 
	is the lift of $T^s_*\mu^u_{\omega^{-,j}}|_{[0:i_1,\ldots, i_s]}$ to  $T^s(W^u_\loc(\omega^{-,j}))\cap[0:i_1,\ldots, i_s]\times \{A_s^E(\omega)v\}$, where again $A_s^E(\omega)v$ stays constant since $\omega\in W^u_\loc(\omega^{-,j})\cap[0;j,i_1,\ldots, i_{s-1},i_s]$. Hence,
	\[
	\frac{1}{m([0:j,i_1,\ldots, i_s]\times\R\bbP^1)}(F^E)^s_*m|_{[0:i_1,\ldots, i_s]}
	\]
	satisfies the assumption of Lemma~\ref{l.cesaro2}. Now for each $\omega\in W^u_\loc(\omega^{-,j})$, we define
	 \[
	 D_i(\omega) := \big([0;\omega_0,\ldots,\omega_{i}] \cap W^u_\loc(\omega^{-,j}) \big)\times \R\bbP^1.
	 \]
	Note that for every integrable function $\psi$, we have
	\[
	\int_{D_{i-1}(T\omega)}\psi(\tilde\omega)\,  d(F^E)^i_*m(\tilde\omega)=\int_{D_i(\omega)}\psi\circ (F^E)^i(\tilde\omega)\, dm(\tilde\omega),
	\]
	where $\tilde\omega$ is the variable of the integrand, which should not be confused with $\omega$. We leave $\tilde\omega$ implicit whenever is it clear from the context. Consider a H\"older continuous function $\varphi\in C^\alpha(\Omega \times \R\bbP^1,\R)$. By Lemma~\ref{l.cesaro2} and the facts described above, given $\varepsilon > 0$, there is $N \ge 1$ such that for every $i\ge 0$ and every $\omega\in\Omega$,
	\begin{equation}\label{e.almostmartingale3}
		\left|\frac1{m(D_i(\omega) )}\int_{D_i(\omega)} \frac1NS_N(\varphi\circ (F^E)^i) \, dm  - \int \varphi \, dm^{u,E} \right| < \frac\varepsilon4.
	\end{equation}
 Let us rewrite this as follows. Define $Y_ i: W^u_\loc(\omega^{-,j}) \to \R$ by
	$$
	Y_i(\omega) = \frac{1}{m(D_i(\omega))} \int_{D_i(\omega)} S_i(\varphi) \, dm,
	$$
	which depends only on $\omega_0,\ldots, \omega_i$. Let $\mathcal{B}_i$ be the $\sigma$-algebra generated by $Y_0 , \ldots , Y_i$, which is basically generated by the cylinder sets $[0;n_0,\ldots, n_i]$. In particular, the conditional expectation of $Y_{i+N}$ respect to $\CB_i$ is 
	\begin{align}\label{eq:ConExpection}
	\nonumber\E\left(Y_{i+N} \Big| \CB_i\right)&=\frac{1}{m(D_i(\omega))}\sum_{\tilde\omega\in D_i(\omega)}m(D_{i+N}(\tilde\omega))Y_{i+N}(\tilde\omega)\\
	&=\frac{1}{m(D_i(\omega))}\sum_{\tilde\omega\in D_i(\omega)}\int_{D_{i+N}(\tilde\omega)}S_{i+N}(\varphi) \,dm\\
	\nonumber&=\frac{1}{m(D_i(\omega))}\int_{D_i(\omega)}S_{i+N}(\varphi) \, dm.
	\end{align}
	Clearly, $\E\left(Y_i\Big|\CB_i\right)=Y_i$. Thus the estimate \eqref{e.almostmartingale3} reads
	\begin{equation}\label{e.almostmartingale}
		\left| \frac{1}{N} \E \left( Y_{i+N} - Y_i \Big| \mathcal{B}_i \right)  - \int \varphi \, dm^{u,E} \right| < \frac\varepsilon4.
	\end{equation}
If we define
	$$
	X_n = Y_{nN} - \sum_{k=1}^n \E \left( Y_{kN} - Y_{(k-1)N} | \mathcal{B}_{(k-1)N} \right),
	$$
	it follows that $\{ X_n \}$ is a martingale, that is, \eqref{e.azumass1} holds for $\{ X_n\}$. Indeed, 
\begin{align}\label{eq:MartingaleDiff}
\nonumber X_{n+1}-X_n&=Y_{(n+1)N}-Y_{nN}-\E\Big(Y_{(n+1)N}-Y_{nN}\Big|\CB_{nN}\Big)\\
&=Y_{(n+1)N}-\E\Big(Y_{(n+1)N}\Big|\CB_{nN}\Big).
\end{align}
Since the $\sigma$-algebra $\CF_n$ generated by $X_1,\ldots, X_n$ is precisely $\CB_{nN}$, the above equation implies that for all $n\ge 1$,
\[
\E\left(X_{n+1}-X_n\Big| \CF_n\right)=0.
\]
We claim that \eqref{e.azumass2} also holds true for $\{ X_n \}$ because of the H\"older continuity of $\varphi$. Indeed, it is clear that $|X_1|=|Y_N-\E(Y_N-Y_0\Big|\CB_0)|=|Y_N-Y_0-\E(Y_N\Big|\CB_0)|<CN$. By \eqref{eq:ConExpection} and \eqref{eq:MartingaleDiff}, $X_{n+1}-X_n$ may be rewritten as
\[
\frac{1}{m(D_{(n+1)N}(\omega))}\int_{D_{(n+1)N}(\omega)}S_{(n+1)N}(\varphi) \, dm - \frac{1}{m(D_{nN}(\omega))}\int_{D_{nN}(\omega)}S_{(n+1)N}(\varphi) \, dm.
\]
Note that $A^E_i(\omega')\vec v$ is independent of $\omega'$ for all $\omega'\in [0;\omega_0,\ldots,\omega_{nN}]\cap W^u_\loc(\omega^{-,j})$ and all $0\le i\le nN$. So we may denote the projection of  such $A^E_i(\omega')\vec v$ in $\R\bbP^1$ by $v_i$. Now for any $\omega', \omega''\in[0;\omega_0,\ldots,\omega_{nN}]\cap W^u_\loc(\omega^{-,j})$, we have 
\begin{align*}
	&\left|S_{(n+1)N}(\varphi)(\omega',v)-S_{(n+1)N}(\varphi)(\omega'',v)\right|\\
	\le&\left(\sum^{nN-1}_{i=0}+\sum^{(n+1)N-1}_{nN}\right)\Big|\varphi\big((F^E)^{i}(\omega',v)\big)-\varphi\big((F^E)^{i}(\omega'',v)\big)\Big|
	\\
	\le&\sum^{nN-1}_{i=0}\Big|\varphi\big((F^E)^{i}(\omega',v)\big)-\varphi\big((F^E)^{i}(\omega'',v)\big)\Big|+CN\\
	\le&\sum^{nN-1}_{i=0}\big|\varphi\big(T^i\omega',v_i\big)-\varphi\big(T^i\omega'',v_i\big)\big|+CN\\
	\le&C\sum^{nN-1}_{i=0}d(T^i\omega',T^i\omega'')^\alpha+CN\\
	\le&C\sum^{nN-1}_{i=0}e^{-\alpha(nN-i)}+CN\\
	\le&C(e^{-\alpha}+N),
	\end{align*}
	where $C$ may vary but is independent of $\omega$, $\omega^{-,j}$, $v$, or $E$. Let $a=C(e^{-\alpha}+N)$. This implies that 
	\[
	\left|X_1\right|<a\mbox{ and }\left|X_{n+1}-X_n\right|<a\mbox{ for all }n\ge 1.
	\]
	Thus, it follows from Lemma~\ref{l.azuma} that for every $\delta>0$ and all $n\ge 1$, we have
	\begin{equation}\label{eq:LDT1}
	\mu^u_{\omega^{-,j}}\left\{\omega:\left|\frac1nX_n\right|>\delta\right\}=\bbP\left(\left|\frac1nX_n\right|>\delta\right)<e^{-\frac{\delta^2}{2a^2}n}.
	\end{equation}
	Suppose $\omega\in W^u_\loc(\omega^{-,j})$ satisfies 
	\begin{equation}\label{eq:largeDevi}
	\left|\frac{1}{nN}S_{nN}(\varphi)(\omega)-\int\varphi\, dm^{u,E}\right|>\varepsilon.
	\end{equation}
	Then by the same argument above bounding $|X_{n+1}-X_n|$, we have for all $\omega'\in [0;\omega_0,\ldots,\omega_{nN}]\cap W^u_\loc(\omega^{-,j})$ that
	\[
	\left|S_{nN}(\phi)(\omega,v)-S_{nN}(\phi)(\omega',v)\right|<C.
	\] 
	Combining \eqref{eq:largeDevi} and the estimate above, we obtain for all $n\ge N_0$ that 
	\[
	\left|\frac1{nN}Y_{nN}(\omega)-\int\varphi\, dm^{u,E}\right|>\frac\varepsilon2,
	\]
	where $N_0$ depends only on $\varphi$ and $\varepsilon$. By the estimate above and \eqref{e.almostmartingale}, we have for all $\omega$ satisfying \eqref{eq:largeDevi} that
	\begin{align*}
		\left|\frac1{nN} X_n\right|
		&=\left|\frac1{nN}Y_{nN} - \frac1{n}\sum_{k=1}^n\frac1N\E \left( Y_{kN} - Y_{(k-1)N} \Big| \mathcal{B}_{(k-1)N} \right)\right|\\
		&\ge \left|\frac1{nN}Y_{nN} -\int\varphi\, dm^{u,E}\right|- \frac1n\sum_{k=1}^n \left|\frac1N\E \left( Y_{kN} - Y_{(k-1)N} | \mathcal{B}_{(k-1)N} \right)-\int\varphi\, dm^{u,E}\right|\\
		&>\frac\varepsilon4.
		\end{align*}
		Applying \eqref{eq:LDT1} to $\delta=\frac{\varepsilon N}{4}$, we obtain for all $n\ge N_0$,
		\begin{align*}
			\mu^u_{\omega^{-,j}} & \left\{\omega:	\left|\frac{1}{nN}S_{nN}(\varphi)(\omega)-\int\varphi\, dm^{u,E}\right|>\varepsilon\right\}\\
			& \le \mu^u_{\omega^{-,j}}\left\{\omega: \left|\frac1n X_n\right|\ge \frac{N\varepsilon}{4}\right\}\\
			& \le e^{-\frac{N\varepsilon^2}{16a^2}nN}.
			\end{align*}
			Absorbing the statement for $1\le n\le N_0$ into some constant $C=C(\varphi,\varepsilon)$, we obtain the desired estimate
			\[\mu^u_{\omega^{-,j}}\left\{\omega:	\left|\frac{1}{nN}S_{nN}(\varphi)(\omega)-\int\varphi\, dm^{u,E}\right|>\varepsilon\right\}<Ce^{-\frac{N\varepsilon^2}{16a^2}nN}
			\]
			along the sequence $\{nN\}$ with all the constants depending only on $\varphi$ and $\varepsilon$. 
		
		Replacing the sequence $\{ nN \}$ above by $ \{ nN + l \}$ for $l = 1, \ldots , N-1$ and running the same argument, we obtain the estimate \eqref{e.azumaconc} for all $n$.
\end{proof}

\section{Establishing Localization -- Proof of Theorem~\ref{t:main}}\label{sec:doubleresonances}

In this section, we derive localization from PLE and ULD, and hence prove Theorem~\ref{t:main}. The overall structure of the proof follows the outline of the treatment of the Bernoulli Anderson case in  \cite{bucaj2}. However, many of the key steps in the proof given in \cite{bucaj2} require the independence of the potential values and hence won't work in the same way in our far more general setting. Consequently, we will give full details whenever the absence of independence requires a new argument, while referring the reader to \cite{bucaj2} for details regarding the steps that extend without changes from the Bernoulli Anderson case to our setting.

For some of the key lemmas we have to use techniques from \cite{ADZ, BS}. 
For example, Lemma~\ref{l.additiveldt+2} generalizes \cite[Lemma 8.1]{BS}, while the proof of Proposition~\ref{prop:doubleRes}, the elimination of double resonances, combines various different techniques from \cite[Section 3]{ADZ}, \cite[Section 9]{BS}, and \cite[Section 6]{bucaj2}. 

\subsection{Bounding the Green Function }In the proof we shall use the H\"older continuity of the Lyapunove exponent, which is a direct consequence of PLE and ULD; see, for example, \cite[Theorem 4.1]{bucaj2}. We formulate it as follows. There exist constants $C > 0$, $\beta > 0$ depending on $f$ and $I$ such that
\begin{equation} \label{eq:holder}
	|L(E)-L(E')|
	\leq
	C|E-E'|^\beta
\end{equation}
for all $E,E' \in I$. First, we define 
\[g_n(\omega,E)=\frac1n\log\|A_n^E(\omega)\|.\] 
Let $\Gamma=\sup_{\omega\in\Omega,E\in I}\|A^E(\omega)\|$. Since $f\in C^\alpha(\Omega,\R)$, a direct computation shows that 
\begin{equation}\label{eq:holdergn}
	\left|g_n(\omega',E')-g_n(\omega,E)\right|<\Gamma^{n-1}\big(Cd(\omega,\omega')^\alpha+|E'-E|\big)
\end{equation}
uniformly for all $E\in I$. Then we have:

\begin{lemma}\label{l:additiveLDT}
	For every $\varepsilon>0$, there is a $n_0=n_0(\varepsilon)$ such that 
	\[
	\mu\Big\{\omega:\Big|L(E)-\frac1r\sum^{r-1}_{s=0}g_n(T^{ns+s_0}\omega,E)\Big|>\varepsilon\Big\}\le e^{-\frac{c\varepsilon^2}{n^2}r}
	\]
	for all $E\in I$, $r\ge C\frac{n}{\e}$, $s_0\in\Z$ and all $n\ge n_0$.
\end{lemma}

Using the techniques of \cite[Section 3]{ADZ}, we can reduce $g_n(\cdot, E)$ to a family of functions in $C^{\frac\alpha2}(\Omega^+,\R)$. We briefly introduce the reduction process since we need a bit more information than what is discussed in \cite[Section 3]{ADZ}. Following the discussion after Definition~\ref{def:indep_of_future}, for each $1\le j\le \ell$, we fix a choice $\omega^{(j)}\in[0;j]$ and consider $\varphi(\omega) = \omega^{(\omega_0)}\wedge \omega\in W^u_\loc(\omega^{(\omega_0)})\cap W^s_\loc(\omega)$. For each $n$, we construct a family of functions $h_n(\omega, E)$ via
\begin{equation}\label{eq:reductionh}
h^s_n(\omega,E) := \sum^{\infty}_{k=0} \big[ g_n(T^k\omega,E) - g_n(T^k\varphi(\omega),E) \big].
\end{equation}
Then 
\begin{equation}\label{eq:conj_to_g+}
	g^+_n(\omega ,E):=g_n(\omega,E)+h_n(T\omega, E)-h_n(\omega, E)
\end{equation}
is constant on every local stable set $W^s_\loc(\omega)$. By the estimate of \cite[Section 3]{ADZ}, we see that 
\[
\big|h(\omega,E)-h(\omega',E)\big|<C'\Gamma^{n-1}d(\omega,\omega')^{\frac\alpha2},
\]
where $C'=\frac{C}{1-e^{-\frac\alpha2}}$ and $C$ is from \eqref{eq:holdergn}. In particular, we may just say that 
\begin{equation}\label{eq:holdergn+}
|g^+_n(\omega, E)-g^+_n(\omega',E)|<C\Gamma^{n-1}d(\omega,\omega')^{\frac\alpha2}.
\end{equation}
Moreover, choosing $N=\lceil\frac{\log\Gamma}{\alpha}n\rceil$, we see that 
\begin{align}\label{eq:bddhn}
	\nonumber |h^s_n(\omega) |&\le  \sum^{N}_{k=0} \big|g_n(T^k\omega) - g_n(T^k\varphi(\omega)) \big|+\big|\sum^\infty_N\big|g_n(T^k\omega) - g_n(T^k\varphi(\omega))\big|\\
	&\le 2\Gamma N+Ce^{-\alpha N}\Gamma^{n}\\
	\nonumber &\le C\Gamma N\\
	\nonumber&\le Cn,
	\end{align} 
	which implies that 
	\begin{equation}\label{eq:bddgn+}
		\sup_{E\in I}\|g^+_n(\cdot, E)\|_{\infty}\le Cn.
		\end{equation}
Then, following the strategy of \cite[Section 3]{ADZ}, Lemma~\ref{l:additiveLDT} will be a consequence of the following lemma:

\begin{lemma}\label{l:additiveLDT+}
	For every $\varepsilon>0$, there is a $n_0=n_0(\varepsilon)$ such that 
	\[
	\mu^+\Big\{\omega^+:\Big|L(E)-\frac1r\sum^{r-1}_{s=0}g^+_n(T^{ns}_+\omega^+,E)\Big|>\varepsilon\Big\}\le e^{-c\frac{\varepsilon^2}{n^2}r}
	\]
	for all $E\in I$, $r\in\Z_+$, and all $n\ge n_0$.
\end{lemma} 

The proof of Lemma~\ref{l:additiveLDT+} relies on Lemma~\ref{l.additiveldt+2} below, whose proof again uses Lemma~\ref{l.azuma}.

\begin{lemma}\label{l.additiveldt+2}
Let $K>1$ and assume that $F$ is $\alpha$-H\"older continuous on $\Omega^+$ with $\|F\|_\infty<1$ and $|F(\omega)-F(\omega')|<Kd(\omega,\omega')^\alpha$. Then for all $\varepsilon>0$ and all $r, n\ge 1$, we have
	\[
	\mu^+\left\{\omega^+:\left|\frac1r\sum^{r-1}_{s=0}F(T^{ns}_+\omega^+)-\int F d\mu^+\right|<\varepsilon \right\}<\exp\left(-\frac{c\varepsilon^2n^2}{\log^2 (K\varepsilon^{-1})}r\right).
	\]
\end{lemma}

\begin{proof}
For $\vec l=(l_1,\ldots, l_n)$, where the finite word $l_1\cdots l_n$ is admissible, we instead just say that $\vec l$ is admissible. Let $|\vec l|$ denotes the number of entries in $\vec l$. For an admissible $\vec l$, we set $\Omega^+_{\vec l}:=[0;\vec l]^+.$

	Fix any $n\ge 1$ and let $\CB_i$ be the $\sigma$-algebra generated by the cylinder sets $\Omega^+_{\vec l}$, where $|\vec l|=ni$. Denote by $\mathbb E_i$ the conditional expectation operator with respect to $\CB_i$. Hence,
	\[\mathbb E_i[F](\omega^+)=\sum_{\vec l}\frac{1}{\mu^+(\Omega^+_{\vec l})}\int_{\Omega^+_{\vec l}}F\, d\mu^+\cdot \chi_{\Omega^+_{\vec l}}(\omega^+),
	\] 
	where the sum is taken over all admissible $\vec l$.
	Note that $\mathbb E_0[F]=\int F\, d\mu^+$ and we let $\Delta_i F=\mathbb E_i[F]-\mathbb E_{i-1}[F]$ for all $i\ge 1$. Clearly, we only need to consider $\varepsilon>0$ small. Thus we may choose an integer $N$ such that
	\[
	\log\frac{10K}{\varepsilon}<\alpha nN<2\log\frac{10K}{\varepsilon}.
	\]
	In particular, we have
	\begin{equation}\label{eq:chooseN}
		Ke^{-\alpha nN}<\frac{\varepsilon}{10}\ \mbox{and}\ \frac1{N}>\frac{cn}{\log\frac K\varepsilon}.
	\end{equation}
	By H\"older continuity of $F$, we have
	\begin{equation}\label{eq:FtoFN}
		\|F-\mathbb E_N[F]\|_\infty<Ke^{-nN\alpha}<\frac\varepsilon{10}.
	\end{equation}
	Denote $\mathbb E_N[F]$ by $F_N$ for simplicity. Let $i\wedge j=\max\{i,j\}$. Note that 
	\[
	F_N=\int F\, d\mu^++\sum^N_{i=0}\Delta_i(F).
	\] 
	Thus, for $r\ge 1$
	\begin{align*}
		\frac{1}{r}\sum^{r-1}_{s=0}F_N(T^{ns}_+\omega^+)&=\int F \, d\mu^+\\
		&+\frac1r\Delta_1F\ &(=d_1)\\
		&+\frac1r[\Delta_1F(T^n_+\omega^+)+\Delta_2F(\omega^+)]\ &(=d_2)\\
		&+\cdots+\frac1r\sum^{i\wedge N-1}_{k=(i-r)\wedge 1}\Delta_kF(T^{n(i-k)}_+\omega^+)+\cdots\ &(=d_i)\\
		&+\frac1r[\Delta_{N-1}F(T^{n(r-1)}_+\omega^+)+\Delta_NF(T^{n(r-2)}_+\omega^+)]\  &(=d_{N+r-2})\\
		&+\frac1r\Delta_{N}F(T^{n(r-1)}_+\omega^+)\ &(=d_{N+r-1}).
	\end{align*}
	It is clear that $\{d_i\}_{i\ge 1}$ is a martingale sequence adapted to $\CB_i$ since $\mathbb E_{i-1}(d_{i})=0$. Moreover, since $\|F\|_\infty\le 1$, it is straightforward to see that for $r\le N$:
	\[
	\|d_i\|_\infty\le \begin{cases} 2\frac{i}{r}, &1\le i< r\\
		2, & r\le i\le N\\
		2\frac{N+r-i}{r}, &N< i\le N+r-1\end{cases},
	\]
	which implies
	\[
	\sum^{N+r}_{i=1}\|d_i\|^2_\infty\le 8\sum^{r-1}_{i=1}\frac{i^2}{r^2}+4(N-r+1)\le 4N\le 4\frac{N^2}{r}.
	\]
	If $r\ge N$, we have 
	\[
	\|d_i\|_\infty\le \begin{cases} 2\frac{i}{r}, &1\le i< N\\
		2\frac{N}{r}, & N\le i\le r\\
		2\frac{N+r-i}{r}, & r<i\le N+r-1
	\end{cases},
	\]
	which implies
	\[
	\sum^{N+r}_{i=1}\|d_i\|^2_\infty\le 8\sum^{N-1}_{i=1}\frac{i^2}{r^2}+4\frac{N^2}{r^2}(r-N+1)\le 4\frac{N^2}{r}.
	\]
	In all cases, we have 
	\begin{equation}\label{eq:bdd_difference}
	\left(\sum^{N+r}_{i=1}\|d_i\|^2_\infty\right)^{\frac12}\le \frac{2N}{\sqrt r}.
	\end{equation}
	By \eqref{e.azumass3}, we have for all $\delta>0$ and $m\in\Z_+$ that
	\[
		\bbP\left[\bigg|\sum^{m}_{1}d_i\bigg|>\delta m^{\frac12}\left(\sum^{m}_{1}\|d_i\|^2_\infty\right)^\frac12\right]<e^{-c\delta^2m}.
	\]
	By setting $m=r+N$ and $\delta m^{\frac12}\left(\sum^{m}_{1}\|d_i\|^2_\infty\right)^\frac12=\frac\e{10}$, and using \eqref{eq:bdd_difference}, we obtain
	\begin{align*}
	\mu^{+} & \left\{\omega^+: \left|\frac{1}{r}\sum^{r-1}_{s=0}F_N(T^{ns}_+\omega^+)-\int F \, d\mu^+\right|>\frac\varepsilon{10}\right\}\\
	& = \bbP\left[\bigg|\sum^{m}_{1}d_i\bigg|>\delta m^{\frac12}\left(\sum^{m}_{1}\|d_i\|^2_\infty\right)^\frac12\right]\\
	& < \exp\bigg(-c\frac{\e^2}{\sum^m_1\|d_i\|^2_\infty}\bigg)\\
	& \le \exp\bigg(-c\frac{\varepsilon^2}{N^2}r\bigg),
	\end{align*}
	and hence, in view of \eqref{eq:chooseN} and \eqref{eq:FtoFN}, we have 
	\[
	\mu^{+}\left\{\omega^+: \left|\frac{1}{r}\sum^{r-1}_{s=0}F(T^{ns}_+\omega^+)-\int F \, d\mu^+\right|>\varepsilon\right\}<\exp\left(-\frac{c\varepsilon^2n^2}{\log^2(\varepsilon^{-1} K)}r\right),
	\]
which is the desired estimate.	
\end{proof}

Now we are ready to prove Lemma~\ref{l:additiveLDT+}.
\begin{proof}[Proof of Lemma~\ref{l:additiveLDT+}]
	By ULD on $I$ and the relation between $g_n(\cdot,E)$ and $g^+_n(\cdot, E)$, for every $\varepsilon>0$, there are $c,C>0$ such that
	\[
	\mu^+\left\{\omega^+: \left|g^+_n(\omega^+, E)-L(E)\right|>\epsilon\right\} < Ce^{-cn},
	\] 
	uniformly for all $E\in I$, which implies the existence of some $n_1(\e)$ so that for all $n\ge n_1(\varepsilon)$,
	\[
	\left|\int g^+_n(\omega^+,E)\, d\mu^+-L(E)\right|<\frac\varepsilon{10}.
	\]
	By \eqref{eq:holdergn+} and \eqref{eq:bddgn+}, we may apply Lemma~\ref{l.additiveldt+2} to $F=\frac{g^+_n}{Cn}$ and $K=C\Gamma^{n}$ to obtain some $n_2(\varepsilon)$ so that for all $n\ge n_2(\varepsilon)$,
	\[
	\mu^+\left\{\omega^+: \left|\frac1r\sum^{r-1}_{s=0} g^+_n(T^{ns}_+\omega^+, E)-\int g^+_n\, d\mu^+\right|>\frac\varepsilon{10}\right\}<e^{-\frac{c\varepsilon^2}{n^2}r}.
	\]
	Combining the two estimates above, we obtain
	\[
	\mu^+\left\{\omega^+: \left|\frac1r\sum^{r-1}_{s=0} g^+_n(T^{ns}_+\omega^+, E)-L(E)\right|>\varepsilon\right\}<e^{-c\frac{\varepsilon^2}{n^2}r},
	\]
	uniformly for all $E\in I$, $r\ge 1$, and $n\ge n_0(\varepsilon)=\max\{n_1(\varepsilon), n_2(\varepsilon)\}$.
\end{proof}

\begin{proof}[Proof of Lemma~\ref{l:additiveLDT}]
	This proof is similar to the proof of \cite[Theorem 3.1]{ADZ}. We copy a part of it here since we need the explicit lower bound on $r$. 
	
Let
	$$  	
	\CB^+_r(\varepsilon) := \bigg\{ \omega^+ \in \Omega^+ : \bigg| \frac{1}{r} \sum^{r-1}_{s=0} g^+_n(T^{ns}_+\omega^+,E) - L(E) \bigg| > \varepsilon \bigg\}.
	$$
	In view of \eqref{eq:conj_to_g+} and \eqref{eq:bddhn}, for $r\ge C\frac{n}\e$, we have
	\[
	\bigg| \frac{1}{r} \sum^{r-1}_{s=0} g_n(T^{ns}\omega,E) - L(E) \bigg| > \varepsilon \Rightarrow \bigg| \frac{1}{r} \sum^{r-1}_{s=0} g^+_n(T^{ns}_+\omega^+,E) - L(E) \bigg| > \frac\varepsilon2,
	\]
	which implies for all $n\ge n_0(\e)$, $r\ge C\frac{n}\e$, and all $E\in I$ that
	\begin{align*}
		\mu \bigg\{ \omega \in \Omega : \bigg| \frac{1}{r} \sum^{r-1}_{s=0} g_n(T^{ns}\omega,E) - L(E) \bigg| > \varepsilon \bigg\} & \le \mu[ (\pi^+)^{-1} \CB^+_n(\varepsilon/2)] \\
	& = \mu^+(\CB^+_n(\varepsilon/2)) \\
	& < e^{-c\frac{\e^2}{n^2} r}.
	\end{align*}
	Finally, the ability to add the shift by $s_0$ in the statement of Lemma~\ref{l:additiveLDT} is a consequence of the $T$-invariance of $\mu$.
\end{proof}

\begin{prop} \label{prop:meanLDTshifted}
	For any $\varepsilon  \in (0,1)$, there exists a subset $\Omega'= \Omega'(\varepsilon ) \subset \Omega$ of full $\mu$-measure such that the following statement holds true. For every $\varepsilon  \in (0,1)$ and every $\omega \in \Omega'(\varepsilon )$, there exists $\widetilde n_0 = \widetilde n_0(\omega,\varepsilon )$ such that
	\begin{equation}\label{eq:partsum1}
		\left|L(E) - \frac{1}{n^4}\sum_{s=0}^{n^4-1} g_{n}\left(T^{s_0 + ns}(\omega),E\right) \right|
		<\varepsilon 
	\end{equation}
	for all $n,s_0 \in \Z$ with   $n \geq \max\left( \widetilde n_0, (\log(|s_0|+1))^{2/3} \right)$ and all $E \in I$.
\end{prop}

\begin{proof}	
	For a given large enough $n$ (we will determine largeness later), we first consider sets $\CB_{n,s_0} = \CB_{n,s_0}(\varepsilon)$ where \eqref{eq:partsum1} fails to hold:
	\begin{align*}
		\CB_{n,s_0}
		:=
		\left\{\omega : \sup_{E \in I}
		\left|L(E) - \frac{1}{n^4}\sum_{s=0}^{n^4-1} g_n(T^{s_0 + ns}\omega,E) \right|
		\geq
		\e \right\}.
	\end{align*}
	
	Let $\kappa=|I|$ and given $0 < \delta \leq \kappa/2$, define the na\"ive grid
	\[
	I_0
	:=
	\left[ I \cap \left( \delta \, \Z \right) \right].
	\]
	It is straightforward to check that $I_0$ is $\delta$-dense in $I$ in the sense that
	\begin{equation} \label{eq:prop82:deltaDense_centered}
		I
		\subseteq
		\bigcup_{t \in I_0} [t-\delta,t+\delta].
	\end{equation}
	Moreover, we may estimate the cardinality of $I_0$ via
	\[
	\# I_0
	\leq
	\frac{\kappa}{\delta}+1\le \frac{2\kappa}{\delta}.
	\]
	Now, fixing $0 < \e < 1$ and taking $\delta = \e (3 \Gamma^{n})^{-1}$, we may produce a finite set $I_0 \subset I$, which is $\e (3 \Gamma^{n})^{-1}$-dense in $I$ in the sense of \eqref{eq:prop82:deltaDense_centered}, with cardinality bounded above by
	\begin{equation} \label{eq:cardinality_centered}
		\#I_0
		\leq
		\frac{6\kappa \Gamma^{n}}{\e}.
	\end{equation}
	If necessary, enlarge $n$ to ensure that
	\begin{equation}\label{eq:holderApplication}
		C \left(\frac{\e}{6\Gamma^n} \right)^\beta
		<
		\frac{\e}{3},
	\end{equation}
	where $C$ and $\beta$ are from \eqref{eq:holder}.
	Then \eqref{eq:holder} and \eqref{eq:holdergn} yield
	\begin{align*}
		\CB_{n,s_0}
		\subset
		\bigcup_{E \in I_0}
		\left\{\omega: \left|L(E) - \frac{1}{n^4}\sum_{s=0}^{n^4-1} g_{n} \left(T^{s_0 + ns}\omega,E \right) \right|
		\geq
		\frac{\e}{3} \right\}.
	\end{align*}
	Consequently, by taking $n$ large enough that $n \geq \max\{n_0(\e/3), C\e^{-1/3}\}$ and \eqref{eq:holderApplication} holds, and using $T$-invariance of $\mu$, we obtain
	\begin{equation} \label{eq:supersmall}
		\mu(\CB_{n,s_0})
		\leq
		\frac{6 \kappa \Gamma^{n}}{\e} e^{-c\e^2 n^{2}}
	\end{equation}
	by Lemma~\ref{l:additiveLDT}. Now, with
	\[
	\CB_{n}
	:=
	\bigcup\limits_{|s_0| \leq e^{n^{3/2}}} \CB_{n, s_0},
	\]
	it is clear that $\mu\left(\CB_{n}\right) \leq e^{-c\e n^2}$ for $n$ large, so $\Omega' := \Omega \setminus \limsup \CB_{n}$ satisfies $\mu(\Omega') = 1$ by the Borel--Cantelli Lemma. Naturally, for each $\omega \in \Omega'$, we can find $\widetilde n_0 = \widetilde{n}_0(\omega, \e)$ large enough that $\omega \notin \CB_n$ whenever $n \ge \widetilde n_0$. In other words,
	\[
	\omega \notin \CB_{n, s_0} \mbox{ whenever } n\geq \widetilde{n}_0(\omega, \e) \mbox{ and } |s_0| \leq e^{n^{3/2}}.
	\]
	Changing the order of $n$ and $s_0$, the statement above implies
	\[
	\omega \notin \CB_{n,s_0}
	\mbox{ whenever } s_0 \in\Z \mbox{ and } n \geq \max\left( \widetilde n_0, (\log(|s_0|+1))^{2/3} \right).
	\]
	By the definition of $B_{n,s_0}$, we obtain the statement of the proposition.
\end{proof}

We are now in a position to estimate the finite-volume Green functions. Before stating the estimate, we fix some notation. Let $\Lambda = [a,b] \cap \Z$ be a finite subinterval of $\Z$, and denote by $P_\Lambda:\ell^2(\Z) \to \ell^2(\Lambda)$ the canonical projection. We denote the restriction of $H_\omega$ to $\Lambda$ by
$$
H_{\omega,\Lambda}
:=
P_\Lambda H_\omega P_\Lambda^*.
$$
For any $E \notin \sigma(H_{\omega,\Lambda})$, define
$$
G_{\omega,\Lambda}^E
:=
(H_{\omega,\Lambda}-E)^{-1}
$$
to be the resolvent operator associated to $H_{\omega,\Lambda}$.  Like $H_{\omega,\Lambda}$, $G_{\omega,\Lambda}^E$ has a representation as a finite matrix; denote its matrix elements by $G_{\omega,\Lambda}^E(m,n)$, that is,
$$
G_{\omega,\Lambda}^E(m,n)
:=
\left\langle \delta_m , G_{\omega,\Lambda}^E \delta_n \right\rangle,
\quad
n,m \in \Lambda.
$$
Additionally, for $N \in \Z_+$, let us define $H_{\omega,N} := H_{\omega,[0,N-1]}$ to be the restriction of $H_\omega$ to the box $\Lambda_N := [0,N-1]\cap \Z.$  We will likewise use the same notation for the associated resolvent $G_{\omega,N}^E$.
\medskip

Using Cramer's rule, we know that
\begin{equation} \label{eq:greenstodet}
	G_{\omega,N}^E(j,k)
	=
	\frac{\det[H_{\omega,j} - E] \det[H_{T^{k+1}\omega, N-k-1} - E]}{\det[H_{\omega,N} - E]}
\end{equation}
for any $0 \leq j \leq k \leq N-1$ and $E \notin \sigma(H_{\omega,N})$, where one interprets $\det[H_{\omega,0} - E] = 1$.

Another relation that will be important in what follows is
\begin{align}
	\label{eq:transfertodet}
	A_N^E(\omega)
	=
	\begin{bmatrix}
		\det(E - H_{\omega,N}) & -\det(E - H_{T\omega,N-1})\\
		\det(E - H_{\omega,N-1}) & -\det(E - H_{T\omega,N-2})
	\end{bmatrix},
	\quad
	N \ge 2.
\end{align}
This is a standard fact, which one may prove inductively.

In particular, since the norm of a matrix majorizes the absolute value of any of its entries, we obtain
\begin{equation} \label{eq:greenBoundsFromTMBounds}
	\left| G_{\omega,N}^E(j,k) \right|
	\leq
	\frac{\|A_{j}^E(\omega)\| \|A_{N-k}^E(T^{k}\omega)\|}{|\det[H_{\omega,N} - E]|}
\end{equation}
for all $0 \leq j \leq k \leq N-1$ by combining \eqref{eq:greenstodet} and \eqref{eq:transfertodet}. Thus, it is straightforward to transform estimates on transfer matrices into estimates on Green functions of truncations of $H_\omega$; to complete our goal of estimating Green functions, we will use Proposition~\ref{prop:meanLDTshifted} to estimate transfer matrix norms, and then apply \eqref{eq:greenBoundsFromTMBounds}.

\begin{coro}\label{cor:LEandGreenEst_general}
	Given $\varepsilon \in (0,1)$ and $\omega \in\Omega'(\varepsilon)$, there exists $\widetilde n_1 = \widetilde n_1(\omega,\varepsilon)$ so that the following statements hold true.  For all $E \in I$, we have
	\begin{equation}\label{eq:FiniteLyapUpperEstimate_general}
		\frac{1}{n}\log
		\norm{A_n^E(T^{s_0}\omega)}
		\leq
		L(E) + 2\varepsilon
	\end{equation}
	whenever $n, s_0 \in \Z$ satisfy $n \geq \max\left( \widetilde n_1, \log^2(|s_0| + 1) \right)$.
	
	Moreover, for all  $n, s_0 \in \Z$ with $n\ge \e^{-1}\max\left( \widetilde n_1, 2\log^2(|s_0|+1)\right)$, we have
	\begin{equation}\label{eq:greenEst_K_general}
		\left| G_{T^{s_0} \omega,n}^E(j,k) \right|
		\leq
		\frac{\exp[(n - |j - k|)L(E) + C_0 \e n]}{|\det[H_{T^{s_0} \omega,n} - E]|}
	\end{equation}
	for all $E \in I \setminus \sigma(H_{T^{s_0} \omega,n})$ and all $j,k\in [0,n)$, where $C_0$ is a constant that only depends on $f$ and $\alpha$.
\end{coro}

\begin{proof}
	Fix $\e \in (0,1)$, $\omega \in \Omega'(\e)$, and $E \in I$. As usual, our estimates are independent of the energy, so we suppress $E$ from the notation. Choose $\widetilde n_1 \in \Z_+$ large enough that
	\begin{equation}
		\label{eq:AveragingCorollary_general_N_Largeness}
		\widetilde n_1
		\geq
		\max\left( \widetilde n_0(\omega,\e)^3, 4\e^{-1}, 15^4 \right), \quad
		\text{and} \quad
		\frac{20 \log\Gamma}{\widetilde n_1^{1/4} - 5} < \e.
	\end{equation}
	Given $n,s_0 \in \Z$ with $n \geq \max\left(\widetilde n_1 ,\log^2(|s_0|+ 1) \right)$, we want to apply Proposition~\ref{prop:meanLDTshifted} with $m = \lceil n^{1/4}\rceil$. Notice that $0 \le m^4 -n \leq 5m^3$. Thus, by submultiplicativity of the matrix norm and unimodularity of the transfer matrices, we have
	\begin{equation}\label{eq:PartialLyapCor}
		\|A_{n}(T^{s_0}\omega)\|
		\leq
		\Gamma^{5m^3}\prod_{s = 0}^{m^3-1} \|A_m(T^{s_0 + sm} \omega)\|.
	\end{equation}
	By our choice of $\widetilde n_1$, we have $m \geq \widetilde n_0$ and $m \geq (\log(|s_0|+1))^{2/3}$. Thus, combining \eqref{eq:PartialLyapCor} with Proposition~\ref{prop:meanLDTshifted}  and \eqref{eq:AveragingCorollary_general_N_Largeness}, a direct computation shows that
	\begin{equation}\label{eq:lognorm_estimate_general}
		\begin{split}
			\frac{1}{n}\log\|A_{n}(T^{s_0} \omega)\|
			& \leq
			\frac{5m^3 \log \Gamma}{n} + \frac{m^4}{n}(L+\e) \\
			& \leq
			L + \frac{10\log\Gamma}{m-5} + \frac{m}{m-5} \e \\
			& \leq
			L + 2\e,
		\end{split}
	\end{equation}
	which yields \eqref{eq:FiniteLyapUpperEstimate_general}.
	
	Now suppose $n\ge \e^{-1}\max \left( \widetilde n_1, 2 \log^2(|s_0|+1)) \right)$ and put $h := \lceil \varepsilon n\rceil$. For $j \ge 0$, we have
	\begin{equation*}
		\norm{ A_j(T^{s_0} \omega) }
		\leq
		\norm{ A_{j + h}( T^{s_0 - h} \omega ) }
		\norm{ [A_h( T^{s_0 - h} \omega )]^{-1} }.
	\end{equation*}
	Clearly $j \ge 0$ and $h \ge \widetilde n_1$. Moreover, by our choice of $\widetilde n_1$ and the relation between $n$ and $s_0$, a direct computation shows $h \geq \log^2(|s_0 - h| + 1)$. 
	Thus we can apply \eqref{eq:lognorm_estimate_general} to estimate the norms on the right hand side and obtain
	\begin{equation}\label{eq:TransferMatrixFromGreenF1}
		\norm{ A_j(T^{s_0} \omega) }
		\leq
		e^{(j + 2h)L(E) + 2\e(j + 2h)}
		\leq
		e^{jL(E) + C_0\varepsilon n},
	\end{equation}
	where $C_0$ is a constant that depends only on $\widetilde \mu$.
	
	Naturally, one can also apply the analysis above to estimate the transfer matrix $A_{n-k}(T^{s_0 + k} \omega)$ with $ j\le k \le n-1$ as well. Using \eqref{eq:AveragingCorollary_general_N_Largeness} and the relationships among $h$, $s_0$ and $k$, a direct computation shows that $\log^2(|s_0 - k + h| + 1)\le h$. Thus, \eqref{eq:lognorm_estimate_general} yields
	\begin{align}
		\nonumber
		\| A_{n-k}(T^{s_0 + k} \omega) \|
		& \leq
		\left\| A_{n-k+h}(T^{s_0 + k - h} \omega) \right\|
		\left\| \left[ A_{h}(T^{s_0 + k - h} \omega)\right]^{-1}\right \| \\
		\label{eq:TransferMatrixFromGreenF2}
		& \leq
		\exp\left[(n-k)L(E) + C_0 \e n\right].
	\end{align}
	Combining equations \eqref{eq:TransferMatrixFromGreenF1} and \eqref{eq:TransferMatrixFromGreenF2} with the observation \eqref{eq:greenBoundsFromTMBounds}, one sees that for a suitable choice of $C_0$,
	\begin{align*}
		|G_{T^{s_0} \omega,n}^E(j,k)|
		& \leq
		\frac{\|A_{j}(T^{s_0}\omega)\| \|A_{n-k}(T^{s_0 + k} \omega)\|}{|\det[H_{T^{s_0} \omega,n} - E)]|} \\
		& \leq
		\frac{\exp[(n - |j - k|)L(E) + C_0 \e n]}{|\det[H_{T^{s_0} \omega,n} - E)]|}
	\end{align*}
	for all $E \in I \setminus \sigma(H_{T^{s_0} \omega,n})$ and all $0 \le j \le k < n$. The case $j \geq k$ follows because $H$ is self-adjoint and $E$ is real.
\end{proof}

\subsection{Elimination of Double Resonances}

For $N \in \Z_+$, we define
\[
\overline{N}
:=
\left\lfloor N^{\log N}\right\rfloor
=
\left\lfloor e^{(\log N)^2} \right\rfloor,
\]
which is a super-polynomially and subexponentially growing function of $N$.

We now introduce the set of double resonances. Given $\e > 0$ and $N \in \Z_+$, let $\mathcal D_N =\mathcal D_N(\e)$ denote the set of all those $\omega \in \Omega$ such that
\begin{equation} \label{eq:doubleRes:GreenLB_centered}
	\|G_{T^s\omega,[-N_1, N_2]}^E\|
	\ge
	e^{K^2}
\end{equation}
and
\begin{equation} \label{eq:doubleRes:FNupperBound_centered}
	\frac1m\log\|A_m(T^{s+r}\omega,E)\|
	\leq
	L(E)-\e
\end{equation}
for some choice of $s\in \Z$, $K \ge \max(N, \log^2(|s|+1))$, $0\le N_1,N_2 \leq K^9$, $E \in I $, $K^{10} \le r \le \overline K$, and $m \in \{K,2K\}$.

\begin{prop} \label{prop:doubleRes}
	For all $0 < \e < 1$, there exist constants $C > 0$ and $\widetilde\eta > 0$ such that
	\[
	\mu(\mathcal D_N(\e))
	\le
	Ce^{-\widetilde\eta N}
	\]
	for all $N \in \Z_+$.
\end{prop}

To prove Proposition~\ref{prop:doubleRes}, we first need the following lemma from \cite{ADZ}. Recall that for an admissible $\vec l=(l_1,\ldots, l_n)$, we have defined  at the begining of proof of Lemma~\ref{l.additiveldt+2} that $\Omega^+_{\vec l}:=[0;\vec l]^+$. We further define
\beq\label{eq:normalized_forward_measure}
 \mu^+_{\vec l}=\frac1{\mu^+\big(\Omega^+_{\vec l}\big)}(T^{|\vec l|}_+)_*\mu^+\big|_{\Omega^+_{\vec l}}.
\eeq
In other words, $\mu^+_{\vec l}$ is the normalized push-forward of $\mu^+$ under the injective map $T^{|\vec l|}_+:\Omega^+_{\vec l}\to \Omega^+$. By definition, we have
\[
\int_{\Omega^+}f\, d\mu^+_{\vec l}=\frac1{\mu^+(\Omega^+_{\vec{l}})}\int_{\Omega^+_{\vec l}}f\circ T_+^{|\vec l|}\, d\mu^+.
\]
If we view $T^{-|\vec l|}_+$ as a map from $T_+^{|\vec l|}(\Omega^+_{\vec l})$ to $\Omega^+_{\vec l}$, we obtain
\beq\label{eq:inverseIntegral}
\mu^+(\Omega^+_{\vec{l}})\int_{T_+^{|\vec l|}(\Omega^+_{\vec{l}})}f\circ T_+^{-|\vec l|}\, d\mu^+_{\vec l}=\int_{\Omega^+_{\vec{l}}}f\, d\mu^+.
\eeq
Note that $\mu^+_{\vec l}$ is concentrated on $T_+^{|\vec l|}(\Omega^+_{\vec l})$.  Then we have \cite[Lemma 3.3]{ADZ}; stated as Lemma~\ref{l:bdd+} below. We point out that although in the statement of Lemma 3.3, \cite{ADZ} assumes the topological mixing property, it is not needed since the proof there did not use it. 

\begin{lemma}\label{l:bdd+}
	Consider a one-sided subshift of finite type $(\Omega^+,T_+,\mu^+)$, where $\mu^+$ has the bounded distortion property. There exists a $C \ge 1$ so that, uniformly for all admissible $\vec l$, we have
	\beq\label{eq:bdd+1}
	\frac{d\mu^+_{\vec l}}{d\mu^+}(\omega^+)\le C \mbox{ for $\mu$-a.e. } \omega^+,
	\eeq
	where $\frac{d\mu^+_{\vec l}}{d\mu^+}$ is the Radon-Nikodym derivative of $\mu^+_{\vec l}$ with respect to $\mu^+$. In particular, we have for all nonnegative measurable functions $f$ and all admissible $\vec l$,
	\beq\label{eq:bdd+2}
	\int f \, d\mu^+_{\vec l}\le  C\int f \, d\mu^+.
	\eeq
\end{lemma}

\begin{proof}[Proof of Proposition~\ref{prop:doubleRes}]
	Define auxiliary ``bad sets'' for fixed $s$ and $K$:
	\begin{equation*}
		\mathcal D_{K,s} = \set{\omega: \eqref{eq:doubleRes:GreenLB_centered}, \eqref{eq:doubleRes:FNupperBound_centered}\text{ are satisfied for some choice of } E,N_1,N_2,r,m \text{ as above}}.
	\end{equation*}
	Fix $\e \in (0,1)$ and begin by noticing that
	\begin{equation}\label{eq:split-into-N-k-sets_centered}
		\mathcal D_{K,s}
		\subset
		\bigcup_{K^{10} \le r \le \overline{K}}\;\;
		\bigcup_{0 \leq N_1,N_2\leq K^{9}}
		\widetilde{\mathcal D}_1(N_1,N_2,r,s) \cup \widetilde{\mathcal D}_2(N_1,N_2,r,s),
	\end{equation}
	where $\widetilde{\mathcal D}_j(N_1,N_2,r,s)$ denotes the collection of all $\omega \in \Omega$ for which there exists $E \in I$ such that \eqref{eq:doubleRes:GreenLB_centered} and \eqref{eq:doubleRes:FNupperBound_centered} hold with $m=jK$. 
	
	We will estimate $\mu(\widetilde{\mathcal D}_1)$. The estimates for $\widetilde{\mathcal D}_2$ are completely analogous. To that end, suppose $\omega\in\widetilde{\mathcal D}_1(N_1,N_2,r,s)$, that is, \eqref{eq:doubleRes:GreenLB_centered} and \eqref{eq:doubleRes:FNupperBound_centered} hold for some $E \in I$. By the spectral theorem, there exists $E_0 \in \sigma(H_{T^s\omega,[-N_1,N_2]})$ with
	\begin{equation}
		\label{eq:doubleRes:DtildeUB}
		|E - E_0|
		\leq
		\left\|G^E_{T^s\omega,[-N_1,N_2]} \right\|^{-1}
		\leq
		e^{-K^{2}}.
	\end{equation}
	On the other hand, choosing $K$ large enough that $C\Gamma^{K} e^{-\alpha K^{2}}\leq \frac{\e}{6}$ and $ Ce^{-\alpha\beta K^{2}} \leq \frac{\e}{6}$ (where $C,\beta$ are from \eqref{eq:holder} and \eqref{eq:holdergn} and the $\alpha$ in the second inequality will be used later), we get
	\begin{align}\label{eq:gKtranfer}
		\nonumber g_{K}(T^{s+r}\omega,E_0)
		& \leq
		g_{K}(T^{s+r}\omega, E) + \frac{\e}{6}\\
		& \leq
		L(E) - \e + \frac{\e}{6}\\
		\nonumber & \leq
		L(E_0) - \frac{2\e}{3},
	\end{align}
	where we have used \eqref{eq:holdergn} in the first line, \eqref{eq:doubleRes:FNupperBound_centered} in the second line, and \eqref{eq:holder} in the final line.
	Thus, when $K$ is large enough, we get
	\begin{equation} \label{eq:doubleRes:DtildeUB_centered}
		\widetilde{\mathcal D}_1(N_1,N_2,r,s)
		\subseteq
		\hat{\mathcal D}_1(N_1,N_2,r,s) \text{ for all } N_1, N_2, r, \text{ and } s,
	\end{equation}
	where $\hat{\mathcal D}_1 = \hat{\mathcal D}_1(N_1,N_2,r,s)$ denotes the set of all $\omega \in \Omega$ such that
	\[
	g_{K}(T^{s+r}\omega,E_0)
	\leq
	L(E_0) - \frac{\e}{2}
	\]
	for some $E_0\in\sigma\left(H_{T^s\omega,[-N_1,N_2]} \right)$. To estimate the measure of $
	\hat{\mathcal D}_1$, by $T$-invariance, we may instead consider 
	\[
	T^s(\hat{\mathcal D}_1)=\bigcup_{E_0\in  \sigma\big(H_{\omega, [-N_1, N_2]}\big)}\big\{\omega: g_K(T^r\omega, E_0)\le L(E_0)-\frac\e2 \big\}.
	\]
	Divide $\Omega$ into $\Omega=\bigcup_{\vec l} \Omega_{\vec l}$, where $\vec l$ is admissible with $|\vec l|=2K^2+1$ and $\Omega_{\vec l}=[-K^2: \vec l]$. Fix some $\omega^{(\vec l)}\in \Omega_{\vec l}$. Then for all $\omega\in\Omega_{\vec l}$, we have $d(\omega,\omega^{(\vec l)})\le e^{-K^2}$. Thus for each $E_0\in \sigma\big(H_{\omega, [-N_1, N_2]}\big)$, there is a $E'\in \sigma\big(H_{\omega^{(\vec l)}, [-N_1, N_2]}\big)$ such that 
	\[
	|E'-E_0|\le \|H_{\omega, [-N_1, N_2]}-H_{\omega^{(\vec l)}, [-N_1, N_2]}\|<Ce^{-\alpha K^2}.
	\]
Running the estimate \eqref{eq:gKtranfer} again, we obtain that for $\omega\in\Omega_{\vec l}$
\[
g_{K}(T^{r}\omega,E_0)<L(E_0)-\frac23\e\Rightarrow g_{K}(T^{r}\omega,E')<L(E')-\frac\e3.
\]
In other words, we have 
	\begin{equation}\label{eq:fix_omega}
		T^s(\hat{\mathcal D}_1)\subset \bigcup_{\vec l} \bigcup_{E'\in  \sigma\big(H_{\omega^{(\vec l)}, [-N_1, N_2]}\big)}\left\{\omega\in \Omega_{\vec l }: g_K(T^r\omega, E')\le L(E')-\frac\e3 \right\}.
	\end{equation}
	Let $S_K(E,\e)=\{\omega: g_K(\omega, E)\le L(E)-\e\}$. Then the set on the right-hand side of \eqref{eq:fix_omega} becomes 
	\[
	\Omega_{\vec l}\bigcap T^{-r}\big(S_K(E',\e/3)\big).
	\] 
	To estimate the measure of the set above, we define the following $s$-locally saturated set,
	\[
	\tilde S(E)=\bigcup_{\omega\in T^{-K^2}S_K(E,\e/3)}W^s_\loc(\omega).
	\]
	For any $\omega'\in W^s_\loc(\omega)$, we have
	$d(T^{K^2}\omega',T^{K^2}\omega)\le e^{-K^2}$, which implies
	\[
		|g_K(T^{K^2}\omega,E)-g_K(T^{K^2}\omega',E)|\le C\Gamma^{K}e^{-\alpha K^{2}}<\frac\e6.
	\]
	For any $\omega'\in \tilde S(E)$, $\omega'\in W^s_\loc(\omega)$ for some $\omega\in T^{-K^2}(S_K(E,\e/3))$. Hence, 
	\begin{align*}
		g(T^{K^2}\omega',E)&\le g(T^{K^2}\omega,E)+\frac\e6\\
		&\le L(E)-\frac\e3+\frac\e6\\
		&=L(E)-\frac\e6,
				\end{align*}
	which implies
\[
	T^{-K^2}S_K(E,\e/3)\subset \tilde S(E)\subset T^{-K^2}S_K(E,\e/6).
\]
   Clearly,
	\begin{align*}
	T^{-K^2}\left(\Omega_{\vec l}\bigcap T^{-r}S_K(E',\e/3)\right)&\subset T^{-K^2}\left(\Omega_{\vec l}\bigcap T^{-r+K^2}\tilde S(E')\right)\\
	&=\left(T^{-K^2}\Omega_{\vec l}\right)\bigcap \left(T^{-r}\tilde S_K(E')\right),
	\end{align*}
	which is $s$-locally saturated. Let 
	$$
	\tilde S^{+}(E)=\pi^+(\tilde S(E)) \mbox{ and } \Omega^+_{\vec l}:=\pi^+\big(T^{-K^2}\Omega_{\vec l}\big)=[0;\vec l]^+.
	$$ 
	It is straightforward to see that for any $s$-locally saturated set $S$, we have $\mu(S)=\mu^+(\pi^+S)$ and $\pi^+T^{-n}S=T^{-n}_+\pi^+S$, which yields
    \begin{align*}
	\mu\left(\big(T^{-K^2}\Omega_{\vec l}\big)\bigcap \big( T^{-r}\tilde S_K(E')\big)\right)
	&=\mu^+\pi^+\left(\big(T^{-K^2}\Omega_{\vec l}\big)\bigcap \big( T^{-r}\tilde S_K(E')\big)\right)\\
	&\le \mu^+\left(\pi^+\big(T^{-K^2}\Omega_{\vec l}\big)\bigcap \pi^+\big( T^{-r}\tilde S_K(E')\big)\right)\\
	&= \mu^+\left(\Omega^+_{\vec{l}}\bigcap T^{-r}_+\big(\tilde S^+_K(E')\big)\right).
   \end{align*}
By Lemma~\ref{l:bdd+}, the three estimates above, $T$- and $T^+$-invariance of $\mu$ and $\mu^+$ respectively, and the fact that $r\ge K^{10}>2K^2+1$, we obtain
	\begin{align*}
		\mu\big(\Omega_{\vec l}\bigcap T^{-r}\big(S_K(E',\e/3)\big)\big)&=\mu\left(T^{-K^2}\big(\Omega_{\vec l}\bigcap T^{-r}S_K(E',\e/3)\big)\right)\\
		&\le \mu^+\left(\Omega^+_{\vec{l}}\bigcap T^{-r}_+\big(\tilde S^+(E')\big)\right)\\&=
		\int_{\Omega^+_{\vec l}} \chi_{T^{-r}_+(\tilde S^+(E'))}\, d\mu^+\\
		&=\int_{\Omega^+_{\vec l}} \chi_{\tilde S^+(E')}\circ T^{r}_+ d\mu^+\\
		&=\mu^+(\Omega^+_{\vec l}) \int_{\Omega} \chi_{\tilde S^+(E')}\circ T^{r-2K^2-1}_+\, d\mu^+_{\vec l}\\
		&\le C\mu^+(\Omega^+_{\vec l})\int_{\Omega}\chi_{\tilde S^+(E')}\circ T^{r-2K^2-1}_+\, d\mu^+\\
		&=C\mu^+(\Omega^+_{\vec l})\int_{\Omega}\chi_{\tilde S^+(E')}d\mu^+\\
		&=C\mu^+(\Omega^+_{\vec l})\cdot \mu^+\big(\tilde S^+(E')\big)\\
		&=C\mu(\Omega_{\vec l})\cdot \mu\big(\tilde S(E')\big)\\
		&\le C\mu(\Omega_{\vec l})\mu (T^{-K^2}S_K(E', \e/6))\\
		&=C\mu(\Omega_{\vec l})\mu (S_K(E', \e/6))\\
		&\le C\mu(\Omega_{\vec l}) e^{-cK},
	\end{align*}
	where the last line follows from ULD. Then, we have
	\begin{align*}
		\mu(\hat{\mathcal D}_1)&=\mu(T^s\hat{\mathcal D}_1)\\
		& \le C\sum_{\vec l}\mu(\Omega_{\vec l})K^9e^{-cK}
		\\
		& \leq
		C \base^{9} e^{-c\base} \\
		& \leq
		Ce^{- \eta_1 \base}.
	\end{align*}
	Thus, we obtain $\mu(\widetilde{\mathcal D}_1(N_1,N_2,r,s)) \leq Ce^{-\eta_1 \base}$ by applying \eqref{eq:doubleRes:DtildeUB_centered}. 
	
	Applying similar reasoning to $\widetilde{\mathcal D}_2$, one can estimate $\mu(\widetilde{\mathcal D}_2(N_1,N_2,r,s) ) \leq Ce^{-\eta_2\base}$.
	
	Putting everything together yields
	\begin{align*}
		\mu(\mathcal D_{K,s})
		& \leq
		\sum_{0 \leq N_1,N_2 \leq \base^{9}} \;
		\sum_{K^{10} \le r \le \overline{K}}
		\left(\mu(\widetilde{\mathcal D}_1(N_1,N_2,r)) + \mu(\widetilde{\mathcal D}_2(N_1,N_2,r)) \right) \\
		& \leq
		C \base^{18} \overline{\base} e^{-\eta_3 \base} \\
		& \leq
		Ce^{-2\widetilde\eta \base}
	\end{align*}
	for some suitable choice of $\widetilde\eta$. Changing the order of $K$ and $s$, we have
	\[
	\mathcal{D}_N
	=
	\bigcup_{s\in\Z}\ \bigcup_{K\ge\max\{N,\log^2(|s|+1)\}}\mathcal{D}_{K,s}
	\subseteq
	\bigcup_{K \ge N} \bigcup_{|s| \le e^{\sqrt K}} \mathcal D_{K,s}.
	\]
	Then, the estimates above yield
	\[
	\mu(\mathcal{D}_N)
	\leq
	\sum_{K \ge N} (2e^{\sqrt K}+1) Ce^{-2\widetilde\eta K}
	\leq
	Ce^{-\widetilde\eta N}
	\]
	for large enough $N$. Adjusting the constants to account for small $N$ concludes the proof.
\end{proof}
Once we have Corollary~\ref{cor:LEandGreenEst_general} and Proposition~\ref{prop:doubleRes}, we can then mimic the process after the proof of Proposition 6.1 of \cite[Section 6]{bucaj2} to prove \eqref{eq:edl}, hence Theorem~\ref{t:main}.

\begin{remark}\label{rem:halflineops}
An analogous proof can be given for half-line operators associated with non-invertible maps. The changes are relatively simple. In particular, modulo these modifications, the arguments presented in this paper will establish the localization result for the doubling map model stated in Corollary~\ref{cor:doubling}. 
\end{remark}

\section{Applications -- Proof of Corollaries~\ref{t:locally_constant} and \ref{t:fiber_bunched}}\label{sec:examples}

The two corollaries can be proved together. In both cases, we may let $J\subset \R$ be a compact interval so that it contains the almost sure spectrum $\Sigma$ and $A^E$ is either locally constant or fiber bunched over $J$. Let $\CF_f$ be the finite set and $J_\eta$, $\eta>0$, be the finite union of compact intervals as described before the statement of Theorem~\ref{t:uld}. In the same description, we know that PLE holds true on $J_\eta$, and by Theorem~\ref{t:uld}, ULD holds true on $J_\eta$. Hence, by Theorem~\ref{t:main} $H_\omega$ has exponential dynamical localization on $J_\eta$ for $\mu$-a.e. $\omega$.

Thus for every $n\in\Z_+$, there is a full measure set $\Omega^{(n)}\subset \Omega$ so that $H_\omega$ has spectral localization for all $\omega\in\Omega^{(n)}$ on $J_{\frac1n}$. Also, for every $E\in\R$, the set 
\[
\Omega_E:=\{\omega: E \mbox{ is not an eigenvalue of } H_\omega \}
\]
has full measure. This is an easy consequence of Oseledec's theorem or it follows directly from \cite{pastur1980}. Now let 
\[
\Omega^*=\left(\bigcap_{n\in\Z_+}\Omega^{(n)}\right)\cap\left(\bigcap_{E\in \CF_f}\Omega_E\right).
\] 
Then it is clear that $\mu(\Omega^*)=1$. Moreover, for each $\omega\in\Omega^*$, $H_\omega$ has pure point spectrum on $\R\setminus \CF_f$, exponentially decaying eigenfunctions for all eigenvalues $z \in \R \setminus \CF_f$, and $\CF_f$ contains no eigenvalue of $H_\omega$. Hence $H_\omega$ exhibits Anderson localization for each $\omega\in\Omega^*$, concluding the proof.

Finally, Corollary~\ref{cor:doubling} is a special case of Corollary~\ref{t:fiber_bunched}, again noting the brief discussion of half-line localization in Remark~\ref{rem:halflineops}, while Corollary~\ref{cor:markov} is a special case of Corollary~\ref{t:locally_constant}.

\end{document}